\documentclass[final]{siamltex}
\usepackage{amsmath}
\usepackage{amssymb}
\usepackage{float}
\usepackage{comment}
\usepackage{graphicx}
\usepackage{hyperref} 
\hypersetup{draft=false, colorlinks=true, linkcolor=black, urlcolor=blue, citecolor=black}
\usepackage{caption}
\usepackage{tikz}
\usepackage{pgf}

\newtheorem{algorithm}[theorem]{Algorithm}
\newtheorem{remark}[theorem]{Remark}

\allowdisplaybreaks

\def \RR {{\mathbb R}}

\def \< {\langle}
\def \> {\rangle}

\def \RN {\mathbb{R}_+^N}
\def \Drm {\mathrm{D}}

\title{Identifying the stored energy of a hyperelastic structure by using an attenuated Landweber method}

\author{
Julia Seydel\thanks{Department of Mathematics, Saarland University, PO Box 15 11 50, 66041 Saarbr\"ucken, Germany ({\tt julia.seydel@math.uni-sb.de}).} \and
Thomas Schuster\thanks{Department of Mathematics, Saarland University, PO Box 15 11 50, 66041 Saarbr\"ucken, Germany ({\tt thomas.schuster@num.uni-sb.de}),
correspondent author.}}

\begin{document}
\maketitle

\begin{abstract}
We consider the nonlinear, inverse problem of identifying the stored energy function of a hyperelastic material from full knowledge of the displacement field as well as from surface sensor measurements. The displacement field is represented as a solution of Cauchy's equation of motion, which is a nonlinear, elastic wave equation. 
Hyperelasticity means that the first Piola-Kirchhoff stress tensor is given as the gradient of the stored energy function. We assume that a dictionary of suitable functions is available and the aim is to recover the stored energy with respect to this dictionary. The considered inverse problem is of vital interest for the development of structural health monitoring systems which are constructed to detect defects in elastic materials from boundary measurements of the displacement field, since the stored energy encodes the mechanical peroperties of the underlying structure.
In this article we develope a numerical solver for both settings using the attenuated Landweber method. 
We show that the parameter-to-solution map satisfies the local tangential cone condition. This result can be used to prove local convergence of the attenuated Landweber method in case that the full displacement field is measured. 
In our numerical experiments we demonstrate how to construct an appropriate dictionary and show that our algorithm is well suited to localize damages in various situations.  
\end{abstract}



\section{Introduction}\label{sec:introduction}
The starting point of our inverse problem is Cauchy's equation of motion for an hyperelastic material
\begin{equation*}
\rho (x) \ddot{u}(t,x) - \nabla\cdot\nabla_{Y}C(x,Ju(t,x)) = f(t,x),\,
\end{equation*}
where $t\in[0,T]$ denotes time and $x\in \Omega\subset\RR^3$ a point in a bounded, open domain $\Omega$ in $\RR^3$. Furthermore, $\rho(x)\in(0,\infty)$ denotes the mass density in $x\in\Omega$ and 
$f:[0,T]\times \Omega\to \RR^3$ is an external body force. The function $C=C(x,Y)$, $Y\in \RR^{3\times 3}$ with $\mathrm{det}\, Y>0$, is called stored energy function and encodes the mechanical properties of the material. 
The derivative $\nabla_{Y}$ is to be understood componentwise and $Ju(t,x)=(\partial_{j}u_{i}(t,x))_{i,j=1,2,3}$ represents the displacement gradient in $t\in[0,T]$ and $x\in\Omega$. We refer to standard textbooks like \cite{CIARLET:88, Holzapfel200003, MARSDEN;HUGHES:83} for detailed introductions and derivations of Cauchy's equation.

Given initial values $u_{0}$ and $u_{1}$ and assuming that we have homogeneous Dirichlet boundary values we end up with
\begin{equation}\label{cauchy-hyper}
\rho (x) \ddot{u}(t,x) - \nabla\cdot \nabla_Y C(x,Ju(t,x)) = f(t,x)\,
\end{equation}
for $t\in[0,T]$ and $x\in\Omega\subset\RR^3$ with 
\begin{eqnarray}
\label{anfangswert1}
u(0,\cdot) &=& u_0\in H^{2}(\Omega,\RR^3),\\[1ex]
\label{anfangswert2}
\dot{u}(0,\cdot) &=& u_1\in H^{1}(\Omega,\RR^3)
\end{eqnarray}
and 
\begin{equation}\label{randwert}
 u(t,x) = 0,\;\;t\in[0,T],\;\;x\in\partial\Omega.
\end{equation}
Equation (\ref{cauchy-hyper}) describes the behavior of hyperelastic materials. 
An example for hyperelastic materials are carbon fibre reinforced composites (CFC). This is why the inverse problem of identifying $C$ from measurements of the displacement field can be very
important for the detection of defects in such structures, the so called structural health monitoring, see \cite{Giu08}. 
Structural health monitoring systems consist of a number of actors and sensors which are applied to the structure's surface. The idea is that the actors generate guided waves propagating through the structure which interact with defects and are subsequently acquired by the sensors. Mathematically this can be described as an inverse problem of determining material properties from boundary data of the displacement field which is a solution of (\ref{cauchy-hyper}). 
In this article we investigate the reconstruction problem of the stored energy function $C(x,Y)$ where we consider two different scenarios for the data acquisition. On the one hand we assume to have as data the full displacement field available and on the other hand we suppose that the data are measured on parts of the boundary. The second scenario is used as mathematical model for sensor measurements.

Because there are numerous publications on inverse identification problems in elastic media for different settings, we only summarize here articles which are associated with the scope of this article. 
A comprehensive overview of various inverse problems in the field of elasticity offers the article \cite{BC:2005}. 
In \cite{BLL:2011} Bourgeois and others have applied and implemented the linear sampling method, introduced by Colton and Kirsch in \cite{CK:1996} for detection of reverberant scatterers, for the isotropic 
Navier Lam\'{e} equation. Because the method represents a possibility to detect defects in isotropic materials and damages, which are described by such a scatterer, it was used in \cite{BL13} 
for the identification of cracks. 
The inverse problem on determining the spatial component of the source term in a hyperbolic equation with time-dependent principal part is investigated in \cite{jiang2015theoretical}. For solving this problem numerically the authors adopt the classical Tikhonov regularization to transform the inverse problem into an output least-squares minimation that can be solved by the iterative thresholding algorithm. 
In \cite{Bal14} the authors developed an algorithm for the quantitative reconstruction of constitutive parameters, namely the two eigenvalues of the elasticity tensor, in isotropic linear elasticity from noisy full-field measurements.  
An algorithm, that guarantees the conservation of the total energy as well as the conservation of momentum and angular momentum is found in \cite{ST92}. 
The reconstruction of an anisotropic elasticity tensor from a finite number of displacement fields for the linear, stationary elasticity equation is represented in \cite{BMU15}. Lechleiter and Schlasche considered in \cite{lechleiterschlasche} the identification of the Lam\'{e} parameters of the second order elastic wave equation from time-dependent elastic wave measurements at the boundary. For this reason they dispose an inexact Newton iteration validating the Fr\'{e}chet differentiability of the parameter-to-solution map 
in terms of \cite{KIRSCH;RIEDER:14}. A semi-smooth Newton iteration for the same problem was implemented in \cite{BoehmUlbrich} also delivering an expression for the Fr\'{e}chet derivative. 
The topic of our considerations is the identification of spatially variable stored energy functions from time-dependent boundary data which has not been investigated so far, to the best of our knowledge. 
In \cite{BINDER;SCHUSTER:15} an algorithm for defect localization in fibre-reinforced composites from surface sensor measurements was proposed using the equations of linear elastodynamics as mathematical model. The key idea of the method is the interpretation of defects as if they were induced by an external volume force. In some sense this article can be seen as the foundation of the method presented here.

We want to specify the inverse problems to be investigated in this article. Inspired by Kal\-ten\-ba\-cher and Lorenzi \cite{Kaltenbacher200711} and following the authors of \cite{SCHUSTER;WOESTEHOFF:14,WOESTEHOFF;SCHUSTER:14} we assume to have a dictionary $\{C_1,C_2,\ldots,C_N\}$ consisting of appropriate functions $C_K=C_K (x,Y)$, $K=1,\ldots,N$, given such that
\begin{equation}\label{con-comb}
 C (x,Y)  = \sum_{K=1}^{N}\alpha_{K} C_{K} (x,Y)
\end{equation}
with nonnegative constants $\alpha_K\geq 0$, $K=1,\ldots, N$. In that way the mentioned dictionary should consist of physically meaningful elements such as polyconvex functions, see \cite{Ball1977, Holzapfel200003}. The wave equation then is then given as
\begin{equation}\label{cauchy-hyper-comb} 
 \rho(x)\ddot{u}(t,x) - \sum_{K=1}^{N}\alpha_{K}\nabla\cdot [\nabla_{Y}C_{K}(x,Ju(t,x))] = f(t,x)
\end{equation}
for $t\in[0,T]$ and $x\in\Omega\subset\RR^3$. We consider at first the following inverse problem.\\[1ex]
\textbf{(IP I)}\quad Given $(f,u_0,u_1)$ as well as the displacement field $\tilde{u}(t,x)$ for $t\in [0,T]$ and $x\in \Omega$, compute the coefficients $\alpha=(\alpha_1,\ldots,\alpha_N)\in\RN$,
such that $\tilde{u}$ satisfies the initial boundary value problem (IBVP) (\ref{cauchy-hyper-comb}), (\ref{anfangswert1})--(\ref{randwert}).\\[1ex]
Denoting by $\mathcal{T} : \Drm (\mathcal{T})\subset \RN \to L^2 (0,T;H^1 (\Omega,\RR^3))$ the \emph{forward operator}, which maps a vector $\alpha\in\Drm (\mathcal{T}) $ to the unique solution of the IBVP (\ref{cauchy-hyper-comb}), (\ref{anfangswert1})--(\ref{randwert}) for $(f,u_0,u_1)$ fixed, 
then (IP I) is represented by the nonlinear operator equation
\begin{eqnarray}\label{allgGlg}
  \mathcal{T} (\alpha) = \tilde{u}\,.   
\end{eqnarray}  
Here, $\Drm (\mathcal{T})$ denotes the domain of $\mathcal{T}$ to be specified in Section 3. In that case we assume to have the full knowledge of the displacemt field. We solve this problem numerically. Since in practical applications measurements usually are only acquired at the structure's surface we consider a further inverse problem,
 \begin{eqnarray}\label{allgGlgQ}
  \mathcal{Q}\mathcal{T} (\alpha) = \tilde{y}\,,   
\end{eqnarray}
where $\mathcal{Q}$ is the observation operator, which maps the solution of (\ref{cauchy-hyper-comb}) to measured data. Thus the observation operator includes the measurement modalities to our mathematical model. Following the article \cite{BINDER;SCHUSTER:15} we want to model $\mathcal{Q}$ having sensors in mind that average the displacement field on small parts of the boundary of the structure $\partial\Omega$. 
For this reason we define $\mathcal{Q}$ by 
\[\mathcal{Q}[u](t) = \bigg(\int\limits_{\partial\Omega}\langle g_{k},u(t)\rangle_{\mathbb{R}^{3}}\,d\xi\bigg)_{k=1,...,l},\]
where $l$ is the number of sensors and $g_{k}$, $k=1,...,l$, are weight functions that display the localization of the particular sensors. We assume that the functions $g_{k}$ for all $k=1,...,l$ have small support on $\partial\Omega$. For more information about this definition of the observation operator see \cite{BINDER;SCHUSTER:15}.
The second inverse problem is defined as follows.\\[1ex]
\textbf{(IP II)}\quad Given $(f,u_0,u_1)$ as well as data $\tilde{y}\in\mathbb{R}^{l}$, compute the coefficients $\alpha=(\alpha_1,\ldots,\alpha_N)\in\RN$,
such that $\tilde{y}$ satisfies (\ref{allgGlgQ}).\\[1ex]
{\it Outline.} Section 2 provides all mathematical ingredients and tools which are necessary to deduce the results of the article. In particular we summarize an existing uniqueness result
for the solution of the IBVP (\ref{cauchy-hyper-comb}), (\ref{anfangswert1})--(\ref{randwert}) (Theorem \ref{theorem21imp}) and results of the article \cite{SEYDEL;SCHUSTER:16}. Section 3 describes the implementation of a numerical solver the inverse problems (IP I), (IP II) using the attenuated Landweber method. In Section 4 we prove the local tangential cone condition for (IP I) and relying on that a convergence result for the attenuated Landweber iteration when applied to (IP I). 
Numerical results for (IP I) and (IP II) are subject of Section 5. Section 6 concludes the article.

\section{Setting the stage}\label{sec:setup}
For the investigations and proofs in this article we need at first some results the most of which are proven in \cite{WOESTEHOFF;SCHUSTER:14} and \cite{SEYDEL;SCHUSTER:16}. 
The first one is a uniqueness result for the IBVP (\ref{cauchy-hyper-comb}), (\ref{anfangswert1})--(\ref{randwert}) from \cite{WOESTEHOFF;SCHUSTER:14}.

In the following we assume that the conditions $C_{K}(x,0)=0$ and $\nabla_{Y}C_{K}(x,0)=0$ are valid for the function $C_{K}:\;\Omega\times\RR^{3\times3}\to\;\RR$ for all $K=1,...,N$ and $x\in\Omega\subset\RR^3$. 
Furthermore we restrict the nonlinearity of all functions $C_{K}$ and hence of $C$ by supposing, that there exist positive constants $\kappa_{K}^{[0]},\;\kappa_{K}^{[1]},\;\mu_{K}^{[0]},\;\mu_{K}^{[1]}$ for $K=1,...,N$, such that 
\begin{equation}\label{bed1}
 \kappa_{K}^{[0]}\|Y\|_{F}^{2} \leq C_{K}(x,Y) \leq \mu_{K}^{[0]}\|Y\|_{F}^{2}
\end{equation}
and 
\begin{equation}\label{bed2}
 \kappa_{K}^{[1]}\|H\|_{F}^{2} \leq \langle\langle H|\nabla_{Y}\nabla_{Y}C_{K}(x,Y)H\rangle\rangle \leq \mu_{K}^{[1]}\|H\|_{F}^{2}
\end{equation}
hold for all $H,Y\in\RR^{3\times3}$ and for $x\in\Omega$ almost everywhere. Let $\|\cdot\|_{F}$ be the Frobenius norm induced by the inner product of matrices 
$$\langle\langle A|B \rangle\rangle:=\mbox{tr}(A^{\top}B)\qquad \mbox{for }A,B\in\RR^{3\times3}.$$
In addition we require the existence and boundedness of higher derivatives of $C_K$ with respect to $Y$. More precisely we assume, that there are constants $\mu_{K}^{[2]},...,\mu_{K}^{[7]}$ for $K=1,...,N$ with 
\begin{equation}\label{beschr1}
 \|\partial_{Y_{pq}}\partial_{Y_{ij}}\partial_{Y_{kl}}C_{K}\|_{L^{\infty}(\Omega\times\RR^{3\times3})}\leq\mu_{K}^{[2]}
\end{equation}
\begin{equation}\label{beschr2}
 \|\partial_{Y_{ab}}\partial_{Y_{pq}}\partial_{Y_{ij}}\partial_{Y_{kl}}C_{K}\|_{L^{\infty}(\Omega\times\RR^{3\times3})}\leq\mu_{K}^{[3]}
\end{equation}
\begin{equation}\label{beschr3}
 \|\partial_{l}\partial_{Y_{kl}}C_{K}\|_{L^{\infty}(\Omega\times\RR^{3\times3})}\leq\mu_{K}^{[4]}
\end{equation}
\begin{equation}\label{beschr4}
 \|\partial_{Y_{ij}}\partial_{l}\partial_{Y_{kl}}C_{K}\|_{L^{\infty}(\Omega\times\RR^{3\times3})}\leq\mu_{K}^{[5]}
\end{equation}
\begin{equation}\label{beschr5}
 \|\partial_{l}\partial_{Y_{ij}}\partial_{Y_{kl}}C_{K}\|_{L^{\infty}(\Omega\times\RR^{3\times3})}\leq\mu_{K}^{[6]}
\end{equation}
\begin{equation}\label{beschr6}
 \|\partial_{Y_{pq}}\partial_{l}\partial_{Y_{ij}}\partial_{Y_{kl}}C_{K}\|_{L^{\infty}(\Omega\times\RR^{3\times3})}\leq\mu_{K}^{[7]}
\end{equation}
for $a,b,i,j,k,l,p,q=1,2,3$ and $K=1,...,N$. Additionally we require
\begin{equation}\label{vertauschen-partial}
 \partial_{Y_{ij}}\partial_{l}\partial_{Y_{kl}}C(x,Y)=\partial_{l}\partial_{Y_{ij}}\partial_{Y_{kl}}C(x,Y)
\end{equation}
for all $i,j,k,l=1,2,3$, which holds true if e.g. $C\in\mathcal{C}^{4}(\Omega\times\RR^{3\times3})$. We suppose that the mapping $Y\to C_{K}(x,Y)$ is three times continuously differentiable for $x\in\Omega$
almost everywhere.\\ 
Furthermore, the set of admissible coefficient vectors $\alpha=(\alpha_{1},...,\alpha_{N})^{\top}\in\RR_+^N$ of the conic combination (\ref{con-comb}) is supposed to be restricted by assuming 
\begin{eqnarray*}
 && \alpha\in \mathcal{C}((\kappa^{[a]})_{a=1,2},(\mu^{[b]})_{b=1,...,7})\\[1ex]
 &:=& \begin{Bmatrix}
      \alpha\in (0,\infty )^{N}:\sum_{K=1}^{N}\alpha_{K}\kappa_{K}^{[a]}\geq \kappa^{[a]},\;\sum_{K=1}^{N}\alpha_{K}\mu_{K}^{[b]}\leq \mu^{[b]}\\[1ex]
      \mbox{ for all}\; a=1,2\;\mbox{ and }\;b=1,...,7
     \end{Bmatrix}.
\end{eqnarray*}
It is easy to see that this set is coupled to the nonlinearity conditions of $C_{K}$ (\ref{bed1})--(\ref{beschr6}) via the constants $\mu^{[b]}$ for $b=1,...,7$.\\
After all we define the following set of admissible solutions $u$ of IBVP (\ref{cauchy-hyper})--(\ref{randwert}). Let be given the constants 
$M_{i}$, $i=0,...,4$. Then we set
\begin{eqnarray}\label{bedA}
 \mathcal{A} &:=& \mathcal{A}(M_{0},M_{1},M_{2},M_{3},M_{4}) \nonumber \\[1ex]
 &:=& \begin{Bmatrix}
      u\in L^{\infty}((0,T)\times\Omega,\RR^3)\cap W^{1,\infty}((0,T),H^{1}(\Omega,\RR^3)):\\[1ex]
      \|\partial_{l}\partial_{j}u\|_{L^{\infty}((0,T),L^{2}(\Omega,\RR^3))} \leq M_{0},\;\|\partial_{l}\dot{u}_{k}\|_{L^{\infty}((0,T)\times\Omega)} \leq M_{1},\\[1ex]
      \|\partial_{l}\partial_{j}\dot{u}_{k}\|_{L^{\infty}((0,T)\times\Omega)} \leq M_{2},\;\|\partial_{l}\partial_{j}u_{k}\|_{L^{\infty}((0,T)\times\Omega)} \leq M_{3},\\[1ex]
      \|\partial_{l}u_{k}\|_{L^{\infty}((0,T)\times\Omega)} \leq M_{4}\;\mbox{for all}\; i,j,k,l=1,2,3
     \end{Bmatrix}.
\end{eqnarray}
Let us notice that this set is only a subset of the set of admissible solutions mentioned in \cite{WOESTEHOFF;SCHUSTER:14} and \cite{SEYDEL;SCHUSTER:16} because of the additional condition 
$\|\partial_{l}u_{k}\|_{L^{\infty}((0,T)\times\Omega)} \leq M_{4}$ for all $i,j,k,l=1,2,3$. $u\in\mathcal{A}$ holds true, if e.g. $\partial\Omega$, $f$, $u_{0}$, $u_{1}$ and $C_{K}$ are sufficiently smooth.\\
Now all necessary conditions are mentioned to prove the following uniqueness result for the solution of the IBVP (\ref{cauchy-hyper-comb}), (\ref{anfangswert1})--(\ref{randwert}) for given
$\alpha\in \mathcal{C}((\kappa^{[a]})_{a=1,2},(\mu^{[b]})_{b=1,...,7})$, which has been presented in \cite{WOESTEHOFF;SCHUSTER:14}.\\

\begin{theorem}[{\cite[Theorem 2.1]{WOESTEHOFF;SCHUSTER:14}}]\label{theorem21imp}
 Let $u$, $\bar{u}$ be two solutions to the initial boundary value problem (\ref{cauchy-hyper-comb}), (\ref{anfangswert1})--(\ref{randwert}) corresponding to the parameters, initial values and right-hand sides $(\alpha, u_{0}, u_{1}, f)$ and $(\bar{\alpha}, \bar{u}_{0}, \bar{u}_{1}, \bar{f})$, respectively.  
 Furthermore, assume that $u,\bar{u}\in\mathcal{A}$. If, in addition, the condition
 \begin{equation}\label{dim}
  \frac{7}{8}\mu < \kappa < \frac{9}{8}\mu
 \end{equation}
 is satisfied for
 \begin{equation}\label{kappamu}
\kappa:=\sum_{K=1}^{N}\alpha_{K}\kappa_{K}^{[1]}\;\;\mbox{and}\;\; \mu:=\sum_{K=1}^{N}\alpha_{K}\mu_{K}^{[1]} 
\end{equation}
and if there are constants $\kappa(\alpha)$ and $\mu(\alpha)$, so that
\begin{equation}\label{kappamuungl}
\kappa\geq\kappa(\alpha)>0\;\; \mbox{and}\;\; \mu\leq\mu(\alpha),
\end{equation}
then there exist constants $\bar{C}_{0}$, 
$\bar{C}_{1}$ and $\bar{C}_{2}$, such that the stability estimate
\begin{eqnarray*}
 &&\bigg[\rho\|(\dot{u}-\dot{\bar{u}})(t,\cdot)\|_{L^{2}(\Omega,\RR^3)}^{2} + \kappa(\alpha)\|(Ju-J\bar{u})(t,\cdot)\|_{L^{2}(\Omega,\RR^{3\times 3})}^{2}\\[1ex]
 && + \rho\|(\ddot{u}-\ddot{\bar{u}})(t,\cdot)\|_{L^{2}(\Omega,\RR^3)}^{2} + \kappa(\alpha)\|(J\dot{u}-J\dot{\bar{u}})(t,\cdot)\|_{L^{2}(\Omega,\RR^{3\times 3})}^{2}\\[1ex]
 &&\;\; +\|(u-\bar{u})(t,\cdot)\|_{H^{2}(\Omega,\RR^3)}^{2}\bigg]^{\frac{1}{2}}\\
 &&\leq \bar{C}_{0}\bigg[\mu(\alpha)\|u_{0}-\bar{u}_{0}\|_{H^{2}(\Omega,\RR^3)}^{2} + \|u_{1}-\bar{u}_{1}\|_{H^{1}(\Omega,\RR^3)}^{2}\bigg]^{\frac{1}{2}}\\[1ex]
 && + \bar{C}_{1} \|f-\bar{f}\|_{W^{1,1}((0,T),L^{2}(\Omega,\RR^3))} + \bar{C}_{2}\|\alpha-\bar{\alpha}\|_{\infty}
\end{eqnarray*}
is valid for all $t\in(0,T)$.
Thereby, the constants $\bar{C}_{0}$, $\bar{C}_{1}$ and $\bar{C}_{2}$ only depend on $T$, $M_{0}$, $M_{1}$, $M_{2}$, $M_{3}$,
\begin{equation}
 \bar{C}(\alpha) := \sum_{K=1}^{N}\alpha_{K}\mu_{K}^{[2]}\bigg(\sum_{K=1}^{N}\alpha_{K}\kappa_{K}^{[1]}\bigg)^{-1}
\end{equation}
and 
\begin{equation}
 \hat{C}(\alpha) := \frac{\hat{K}}{1-\sqrt{1-\epsilon}}\sum_{K=1}^{N}\alpha_{K}\mu_{K}^{[1]}\bigg(\sum_{K=1}^{N}\alpha_{K}\kappa_{K}^{[1]}\bigg)^{-2},
\end{equation}
where $0<\epsilon<1$ is a constant, whose existence is ensured by inequality (\ref{dim}). The constant $\hat{K}>0$ is defined by the continuity of the embedding $H_{0}^{2,1}(\Omega,\RR^3):=H^{2}(\Omega,\RR^3)\cap H_{0}^{1}(\Omega,\RR^3)\hookrightarrow H^{2}(\Omega,\RR^3)$,
\[\|g\|_{H^{2}(\Omega,\RR^3)}\leq \hat{K}\|g\|_{H_{0}^{2,1}(\Omega,\RR^3)}=\hat{K}\bigg(\sum_{k=1}^{3}\int\limits_{\Omega}\sum_{l=1}^{3}\sum_{j=1}^{3}(\partial_{i}\partial_{j}g_{k}(x))^{2}dx\bigg)^{\frac{1}{2}}\]
for all $g\in H_{0}^{2,1}(\Omega,\RR^3)$. Moreover, the constants $\bar{C}_{0}$, $\bar{C}_{1}$ and $\bar{C}_{2}$ are uniformly bounded, if we take $(M_{0},M_{1},M_{2},M_{3},\bar{C}(\alpha),\hat{C}(\alpha),T)\in\mathcal{M}$ with $\mathcal{M}\subset(0,\infty)^{7}$ bounded.\\[2mm]
\end{theorem}
The function $\bar{C}$ is positive and bounded in the following way because of the non-negativity of the coefficients $\alpha_{K}$: 
\begin{equation}\label{Cbar}
 0 < \zeta := \frac{\min_{1\leq K\leq N}\mu_{K}^{[2]}}{\max_{1\leq K\leq N}\kappa_{K}^{[1]}}\leq \bar{C}(\alpha) \leq \frac{\max_{1\leq K\leq N}\mu_{K}^{[2]}}{\min_{1\leq K\leq N}\kappa_{K}^{[1]}} =: \eta < \infty. 
\end{equation}
Next we prove an estimate being necessary for the proof of Theorem \ref{kegelbed} and following from Theorem \ref{theorem21imp}. 
\begin{corollary}\label{M4}
Let be $u,\bar{u}\in\mathcal{A}(M_{0},M_{1},M_{2},M_{3},M_{4})$ two solutions of the problem (\ref{cauchy-hyper}) - (\ref{randwert}). Then there is for $\tilde{u}=u-\bar{u}$ and all $t\in(0,T)$ 
\begin{eqnarray}\label{M4Bed}
\sum\limits_{i,j=1}^{3}\int\limits_{\Omega}|\partial_{j}\tilde{u}_{i}(t,x)|^{4}\, dx \leq (M_{4})^{2}\|\tilde{u}(t,\cdot)\|_{H_{0}^{1}(\Omega,\RR^3)}^{2}.
\end{eqnarray}
\end{corollary}
\begin{proof}
Using the condition $u,\bar{u}\in\mathcal{A}(M_{0},M_{1},M_{2},M_{3},M_{4})$ it follows by the triangle inequality
\[\|\partial_{j}\tilde{u}_{i}\|_{L^{\infty}((0,T)\times\Omega)}=\|\partial_{j}u_{i}-\partial_{j}\bar{u}_{i}\|_{L^{\infty}((0,T)\times\Omega)}\leq \|\partial_{j}u_{i}\|_{L^{\infty}((0,T)\times\Omega)}+\|\partial_{j}\bar{u}_{i}\|_{L^{\infty}((0,T)\times\Omega)}\leq2M_{4}\]
for all $i,j=1,2,3$. Hence we have $|\partial_{j}\tilde{u}_{i}(t,x)|/(2M_{4})\leq 1$ and therefore
\[\bigg(\frac{|\partial_{j}\tilde{u}_{i}(t,x)|}{2M_{4}}\bigg)^{4}\leq \bigg(\frac{|\partial_{j}\tilde{u}_{i}(t,x)|}{2M_{4}}\bigg)^{2}\]
for all $i,j=1,2,3$ and $(t,x)\in(0,T)\times\Omega$. Finally we obtain the estimate
\begin{eqnarray*}
&&\sum\limits_{i,j=1}^{3}\int\limits_{\Omega}|\partial_{j}\tilde{u}_{i}(t,x)|^{4}\, dx 
=\sum\limits_{i,j=1}^{3}(M_{4})^{4}\int\limits_{\Omega}\bigg(\frac{|\partial_{j}\tilde{u}_{i}(t,x)|}{M_{4}}\bigg)^{4}\, dx \\
&\leq&\sum\limits_{i,j=1}^{3}(M_{4})^{4}\int\limits_{\Omega}\bigg(\frac{|\partial_{j}\tilde{u}_{i}(t,x)|}{M_{4}}\bigg)^{2}\, dx 
=(M_{4})^{2}\|\tilde{u}(t,\cdot)\|_{H_{0}^{1}(\Omega,\RR^3)}^{2}
\end{eqnarray*}
and hence the assertion of the corollary.\hfill
\end{proof}
\\[1ex]

Next we define for the remainder of the article the following spaces
\[ V:=H^{1}(\Omega,\RR^3)\;\; \mbox{ and }\;\; H:=L^{2}(\Omega,\RR^3)\]
and identify $H$ with its dual space $H'$. Then we obtain the Gelfand triple 
\[V \subset H=H' \subset V'\]
with dense, continuous embeddings. Furthermore let be 
\[U:= H_{0}^{1}(\Omega,\RR^3)\]
with 
\[\|u\|_{U}:=\|Ju\|_{L^{2}(\Omega,\RR^{3\times 3})}\]
for all $u\in U$ and thereby 
\[U \subset V \subset H=H' \subset V' \subset U'\]
with dense, continuous embeddings and $U'\simeq H^{-1} (\Omega,\RR^3)$. Then the Poincar\'{e} inequality yields that there is a constant $C_{\Omega}>0$ with 
\begin{equation}\label{poincare-1}
 \|u\|_{V} \leq \sqrt{1+C_{\Omega}}\|u\|_{U}
\end{equation}
for all $u\in U$. Besides it follows directly from the definition of the norms in $H$ and $V$
\begin{equation}\label{L2H1}
\|u\|_{H}\leq \|u\|_{V}
\end{equation}
for all $u\in V$.
\\
For the proof of the local tangential cone condition we need Gronwall's lemma.\\

\begin{lemma}[Gronwall's lemma]
 Let $\psi\in\mathcal{C}(0,T)$ and $b,k\in L^{1}(0,T)$ be non-negative functions. If $\psi$ satisfies
 \begin{equation*}
  \psi(\tau) \leq a + \int\limits_{0}^{\tau}b(t)\psi(t)dt + \int\limits_{0}^{\tau}k(t)\psi(t)^{p}dt
 \end{equation*}
 for all $\tau\in[0,T]$ with constants $p\in(0,1)$ and $a\geq 0$, then
 \begin{equation}\label{gronwall}
  \psi(\tau) \leq \exp\Big(\int\limits_{0}^{\tau}b(t)dt\Big)\bigg[a^{1-p}+(1-p)\int\limits_{0}^{\tau}k(t)\exp((p-1)\int\limits_{0}^{t}b(\sigma)d\sigma)dt\bigg]^{1/(1-p)}
 \end{equation}
 is valid for all $\tau\in[0,T]$.\\
\end{lemma}

A proof of this version is given in \cite{Bainov199205}.
The following results in connection with the G\^{a}teaux/Fr\'{e}chet derivative of the forward operator $\mathcal{T}$ and its adjoint operator are proven in \cite{SEYDEL;SCHUSTER:16}. That is the reason why there are no proofs in the remainder of this section. 
For this purpose let be the domain $\Drm (\mathcal{T})\subset \RN$ of the forward operator $\mathcal{T}$ defined by 
\begin{equation}\label{domain-T}
\Drm (\mathcal{T}) := \big\{ \alpha \in \RN :  \mbox{the IBVP (\ref{cauchy-hyper-comb}), (\ref{anfangswert1})--(\ref{randwert}) has a unique solution } u\in \mathcal{A} \big\}.
\end{equation}
At first we characterize the G\^{a}teaux derivative of $\mathcal{T}$.\\

\begin{lemma}\label{frechet-system}
Let $\alpha\in \mathrm{int} \big(\Drm(\mathcal{T})\big)$ be an interior point of $\Drm(\mathcal{T})$. The G\^{a}teaux derivative $v=\mathcal{T}'(\alpha)h$ of the solution operator $\mathcal{T}$ fulfills for $h\in \RN$ the following linear system of differential equations with homogeneous initial and boundary value conditions
\begin{equation}\label{gateaux-diff}
  \rho\ddot{v}(t,x)-\nabla\cdot [\nabla_{Y}\nabla_{Y}C_{\alpha}(x,Ju(t,x)):Jv(t,x)] = \nabla \cdot [\nabla_{Y}C_{h}(x,Ju(t,x))]
\end{equation}
for $t\in[0,T]$ and $x\in\Omega\subset\RR^{3}$,
\begin{equation}\label{vAnfangswert}
  v(0,x)=\dot{v}(0,x)=0 \qquad \mbox{for } x\in\Omega
\end{equation}
and 
\begin{equation}\label{vRandwert}
 v(t,x)=0\qquad \mbox{for } x\in\partial\Omega.
\end{equation}
Here, we used the notations
\[C_{\alpha}=\sum_{K=1}^{N}\alpha_{K}C_{K}\;\;\mbox{ respectively }\;\;C_{h}=\sum_{K=1}^{N}h_{K}C_{K}.\]
\end{lemma}

It was proven in \cite{SEYDEL;SCHUSTER:16} that the IBVP (\ref{gateaux-diff})--(\ref{vRandwert}) has a unique solution and hence the G\^{a}teaux derivative is well defined.\\

\begin{theorem}\label{gateaux-ex}
 The IBVP (\ref{gateaux-diff})--(\ref{vRandwert}) has a unique, weak solution $v=\mathcal{T}'(\alpha)h$ in $L^{2}(0,T;V)$.\\
\end{theorem}

The aim of \cite{SEYDEL;SCHUSTER:16} was to show, that $\mathcal{T}: \Drm (\mathcal{T})\subset \RN\to L^{2}(0,T;V)$ even is Fr\'{e}chet differentiable. 
To this end it was proven that the mapping $\mathcal{T}' (\alpha) : \RR^N\to  L^{2}(0,T;V)$, $h\mapsto v$, is linear and bounded.
 It is linear because (\ref{gateaux-diff}) is linear in $v$ and $h$. The continuity was subject of the following theorem.\\
 
\begin{theorem}[{\cite[Prop. 3.4]{SEYDEL;SCHUSTER:16}}]\label{gateaux-stetig}
 Adopt the assumption of Lemma \ref{frechet-system}. The G\^{a}teaux derivative $v=\mathcal{T}'(\alpha)h$ is continuous in $h$ for all $h\in\RR^{N}$, i.e. there is a constant $L_{1}>0$ with $\|v\|_{L^{2}(0,T;V)}\leq L_{1}\|h\|_{\infty}$.\\
 \end{theorem}

Next we state the Fr\'{e}chet differentiability of $\mathcal{T}$.\\

\begin{theorem}[{\cite[Th. 3.8]{SEYDEL;SCHUSTER:16}}]\label{thm-glm-konv}
Adopt the assumptions of Lemma \ref{frechet-system} and let $v=\mathcal{T}' (\alpha)h$. There is a constant $L_{2}>0$ depending only on $\Omega$, $T$ and $\alpha$ such that 
\begin{equation}\label{glm-konv-ungl}
  \|u(\alpha+h)-u(\alpha)-v\|_{L^{2}(0,T;V)}\leq L_{2}\|h\|_{\infty}^{\frac{3}{2}}\qquad \mbox{for }\|h\|_{\infty}\rightarrow 0.\vspace{2mm}
\end{equation}
\end{theorem}

We continue by recapitulating the representation of the adjoint operator 
of the Fr\'{e}chet derivative $\mathcal{T}'(\alpha)^*$, which was also proven in \cite{SEYDEL;SCHUSTER:16}. We will see that the adjoint is important when applying iterative solvers as the Landweber method to the inverse
problem $\mathcal{T}(\alpha) = u^{meas}$.
Let be $\mathcal{X}$ the space consisting of all solutions of 
\begin{equation}\label{Bveq}
 Bv = f
\end{equation}
for $f\in L^2 (0,T;H)$, if we define the mapping $B: L^{2}(0,T;U)\cap H^{1}(0,T;H) \to H^{-1} (0,T;H)$ by
\begin{equation}\label{Bv}
 Bv := \rho\ddot{v} - \nabla\cdot [\nabla_{Y}\nabla_{Y}C_{\alpha}(x,Ju):Jv]
\end{equation}
for all $v\in L^{2}(0,T;U)\cap H^{1}(0,T;H)$. Then we have $\mathcal{X} = B^{-1} (L^2 (0,T;H))\subset (L^{2}(0,T;U)\cap H^{1}(0,T;H))$ and the mapping $B: \mathcal{X}\to L^2 (0,T;H)$ is bijective
since (\ref{Bveq}) with (\ref{vAnfangswert}) and (\ref{vRandwert}) is uniquely solvable according to Theorem \ref{gateaux-ex}. This yields $B^{-1} : L^2 (0,T;H)\to \mathcal{X}$, $f\mapsto v$, where $v$ solves (\ref{Bveq}) with (\ref{vAnfangswert}) and (\ref{vRandwert}).
Finally $\mathcal{X}$ endowed with the norm $\|v\|_{\mathcal{X}}=\|B v\|_{L^{2}(0,T;H)}$ turns into a Hilbert space, which is a closed subspace of $L^{2}(0,T;U)\cap H^{1}(0,T;H)$. The following lemma states that the embedding $\mathcal{X} \hookrightarrow L^{2}(0,T;U)\cap H^{1}(0,T;H)$ even
is continuous.\\

\begin{lemma}\label{stetigeinb}
The embedding $\mathcal{X} \hookrightarrow L^{2}(0,T;U)\cap H^{1}(0,T;H)$ is continuous, i.e. there is a constant $C>0$ not depending on $v$ such that
\[ \|v\|_{L^{2}(0,T;U)\cap H^{1}(0,T;H)}\leq C\|v\|_{\mathcal{X}}\,,\qquad v\in \mathcal{X}\,.    \]
\end{lemma}

A representation of the adjoint operator $\mathcal{T}'(\alpha)^* : \mathcal{X}' \to \RR^N$ for fixed $\alpha\in \RN$ is presented in the following theorem.\\

\begin{theorem}\label{Tadjungiert}
Let $\alpha\in (\mathrm{int} \Drm (\mathcal{T}))\subset \RN$ be fixed and $w\in \mathcal{X}'$.
The adjoint operator of the Fr\'{e}chet derivative $\mathcal{T}' (\alpha) : \RR^N \to \mathcal{X}$ is given by 
\begin{equation}
 \mathcal{T}'(\alpha)^{*}w = \Big[-\int\limits_{0}^{T}\int\limits_{\Omega}\nabla_{Y}C_{K}(x,Ju(t,x)):J(B^{-1})^{*}w(t,x) \,dx\,dt\Big]_{K=1,...,N}\in\mathbb{R}^{N},
\end{equation}
where $p:=(B^{-1})^{*}w$ is the weak solution of the hyperbolic, backward IBVP
\begin{eqnarray}
&&\hspace*{-3mm}\label{backward-1} \rho\ddot{p}(t,x)-\nabla\cdot [\nabla_{Y}\nabla_{Y}C_{\alpha}(x,Ju(t,x)):Jp(t,x)]=w(t,x), \\[1ex]
&&\hspace*{-3mm}\label{backward-2} p(T,x)=\dot{p}(T,x)=0,\qquad x\in\Omega, \\[1ex]
&&\hspace*{-3mm}\label{backward-3} p(t,\xi)=0,\qquad (t,\xi)\in[0,T]\times\partial\Omega.
\end{eqnarray}
\end{theorem}

We have now all ingredients together for the proof of the convergence result in Section \ref{sec:convergence} and the numerical solution of (IP I). The last point to consider in this section is the observation operator, which is needed for (IP II). 
Let be $\mathcal{Q}:\;L^{2}(0,T;U)\rightarrow L^{2}(0,T;\mathbb{R}^{l})$ the observation operator with
  \begin{eqnarray}\label{beobachtung}
   \mathcal{Q}[u](t) = \bigg(\int\limits_{\partial\Omega}\langle g_{k},u(t)\rangle_{\mathbb{R}^{3}}\,d\xi\bigg)_{k=1,...,l} = \bigg(\langle g_{k},\Gamma(u(t))\rangle_{L^{2}(\partial\Omega,\mathbb{R}^{3})}\bigg)_{k=1,...,l}\in\mathbb{R}^{l}, 
  \end{eqnarray}
  where $\Gamma:\;V\rightarrow H^{\frac{1}{2}}(\partial\Omega,\RR^3)$ is the trace operator and $g_{k}\in L^{2}(\partial\Omega,\mathbb{R}^{3})$ given weight functions (see Section \ref{sec:introduction}). 
  Then $L^{2}(0,T;\mathbb{R}^{l})$ represents the data space. Because of (\ref{beobachtung}) it is easy to see that $\mathcal{Q}$ is linear in $u\in L^{2}(0,T;U)$. 
  Furthermore it is proven in \cite{BINDER;SCHUSTER:15} that $\mathcal{Q}$ is continuous in $u\in L^{2}(0,T;U)$ and that the adjoint operator $\mathcal{Q}^{*}$ has for all $a\in L^{2}(0,T;\mathbb{R}^{l})$ the representation
  \begin{eqnarray}\label{adjbeobachtung}
   \mathcal{Q}^{*}a = \sum\limits_{k=1}^{l}a_{k}\Gamma^{*}g_{k}.
  \end{eqnarray}

\section{Numerical scheme}\label{sec:scheme}
In this section we outline a numerical solution scheme for computing a solution of (\ref{allgGlg}) with input data $\tilde{u}$. We want to use the attenuated Landweber method (see for example \cite{DeuflhardEnglScherzer98}, \cite{scherzer1998iterative} or \cite{kaltenbacher2008iterative}). 
This iteration is defined for solving (IP I) via
\begin{eqnarray}\label{Landweberalpha}
   \alpha^{(k+1)} = \alpha^{(k)} - \omega \mathcal{T}'(\alpha^{(k)})^{*}(\mathcal{T}(\alpha^{(k)}) - y),\;\; k=0,1,2,...
  \end{eqnarray}
  and for solving (IP II) via
\begin{eqnarray}\label{LandweberalphaQ}
   \alpha^{(k+1)} = \alpha^{(k)} - \omega \mathcal{T}'(\alpha^{(k)})^{*}Q^{*}(Q\mathcal{T}(\alpha^{(k)}) - y),\;\; k=0,1,2,...\,
  \end{eqnarray}
with initial value $\alpha^{(0)}\in D(\mathcal{T})$ and relaxation parameter 
\begin{equation}\label{parameter}
   \omega\in \bigg(0, \frac{1}{C^{2}}\bigg)
  \end{equation}
  with 
  \begin{equation}\label{parameter1}
  C:=\sup\{\|\mathcal{T}'(\alpha)\|: \alpha\in\mathcal{B}_{\rho}(\alpha^{(0)})\}.
  \end{equation}
  We always denote by $\mathcal{B}_{\rho}(\alpha^{(0)})$ the closed ball centered about $\alpha^{(0)}$ with radius $\rho>0$. In (\ref{Landweberalpha}) respectively (\ref{LandweberalphaQ}) we can see that we have to solve in every iteration step of the attenuated Landweber method respectively once the forward problem (\ref{cauchy-hyper-comb}), (\ref{anfangswert1})--(\ref{randwert}) and the adjoint problem (\ref{backward-1}), (\ref{backward-2})--(\ref{backward-3}). 
  So we need at first an algorithm for solving the forward problem. Because of the second time derivative in the differential equation we split initially the differential equation and then we get with $r(t,x):=\dot{u}(t,x)$ for all $(t,x)\in[0,T]\times\Omega$ the equivalent formulation
  \begin{eqnarray}\label{splitCauchy}
 \begin{cases}
    \dot{u}(t,x) - r(t,x) = 0 \\
    \rho\dot{r}(t,x) - \nabla\cdot\nabla_{Y}C(x,J u(t,x)) = f(t,x)
    \end{cases}
 \end{eqnarray}
  for all $(t,x)\in[0,T]\times\Omega$. We discretized this equation system at first in time with the $\theta$-method with $\theta\in[0,1]$. Let $I=[0,T]$ be equidistantly partitioned in $m>0$ equal time steps of length $k=T/m$ and points $t_{j}=jk$, $j=0,...,m$. Then we get with $u^{i}=u(t_{i})$ for all $i=0,...,m$ in every time step $j$ for all $j=1,...,m$ the following formulation of the time discretized equations
  \[
   \begin{cases}
    \frac{u^{j}-u^{j-1}}{k} = \theta r^{j} + (1-\theta)r^{j-1}\\
    \rho\frac{r^{j}-r^{j-1}}{k} = \nabla\cdot\nabla_{Y}C(x,\theta J u^{j} + (1-\theta)J u^{j-1}) + \theta f^{j} + (1-\theta)f^{j-1}.
   \end{cases}
  \]
  After some reformulations we get the system 
  \begin{eqnarray}\label{zeitdiskr}
   \begin{cases}
   u^{j} = u^{j-1} + kr^{j-1} + \frac{k^{2}\theta}{\rho}\nabla\cdot\nabla_{Y}C(x,\theta J u^{j} + (1-\theta)J u^{j-1})\\
   \;\;\;\;\;\;\;\; + \frac{k^{2}\theta^{2}}{\rho}f^{j} + \frac{k^{2}\theta(1-\theta)}{\rho}f^{j-1}\\
   r^{j} = r^{j-1} + \frac{k}{\rho}\nabla\cdot\nabla_{Y}C(x,\theta J u^{j} + (1-\theta)J u^{j-1}) + \frac{k\theta}{\rho}f^{j} + \frac{k(1-\theta)}{\rho}f^{j-1}.
   \end{cases}
  \end{eqnarray}
  It is easy to see that the first equation is not linear in $u^{j}$ and the second one is linear in $r^{j}$. That is the reason why we have to implement a nonlinear solver for the first equation. Then we can discretize in space and solve both equations. For using the Newton method to solve the first equation we define 
   \begin{eqnarray}\label{FNewton}
   F(u_{l}^{j}) &=& u_{l}^{j} - u^{j-1} - kr^{j-1} - \frac{k^{2}\theta}{\rho}\nabla\cdot\nabla_{Y}C(x,\theta Ju_{l}^{j} + (1-\theta) Ju^{j-1}) \nonumber \\
   &&- \frac{k^{2}\theta^{2}}{\rho}f^{j} -\frac{k^{2}\theta(1-\theta)}{\rho}f^{j-1}
  \end{eqnarray}
  with the iteration index $l=0,1,..$ of the Newton method. Then the first equation in (\ref{zeitdiskr}) is equal to the nonlinear equation $F(u_{l}^{j})=0$. The Newton method yields a solution of this equation determining 
  $\partial u_{l}^{j}$ with 
  \[ F'(u_{l}^{j})\partial u_{l}^{j} = - F(u_{l}^{j})\]
  with
  \begin{eqnarray}\label{FAblFormel}
   F'(u_{l}^{j})\partial u_{l}^{j} = \partial u_{l}^{j} - \frac{k^{2}\theta^{2}}{\rho}\nabla\cdot[\nabla_{Y}\nabla_{Y}C(x,\theta J u_{l}^{j} + (1-\theta)J u^{j-1}):J(\partial u_{l}^{j})] 
  \end{eqnarray}
  and after that setting
  \[u_{l+1}^{j} = u_{l}^{j} + \partial u_{l}^{j}\]
  for all $l=0,1,..$ with $u_{0}^{j}=u^{j-1}$ until a sufficient accuracy is achieved. Before we discretize the reformulated time discretized equations we want to present the following weak formulation of the splitted representation of the problem including the nonlinear solver in every time step $t_{j}=jk$. Therefore we choose $H_{0}^{1}(\Omega,\RR^3)$ as solution and test space. 
  Then we get:\\
  \vspace{4mm}
   
   \textit{Find $\partial u_{l}^{j}\in H_{0}^{1}(\Omega,\mathbb{R}^{3})$ such that
   \begin{eqnarray}\label{schwachNewton} 
   \langle F'(u_{l}^{j})\partial u_{l}^{j},\varphi\rangle_{L^{2}(\Omega,\mathbb{R}^{3})} = - \langle F(u_{l}^{j}),\varphi\rangle_{L^{2}(\Omega,\mathbb{R}^{3})}
   \end{eqnarray}
   for all $\varphi\in H_{0}^{1}(\Omega,\mathbb{R}^{3})$ and 
   set
   \[u_{l+1}^{j} = u_{l}^{j} + \partial u_{l}^{j}\]
   for all $l=0,1,..$. \\
   Find $r^{j}\in H_{0}^{1}(\Omega,\mathbb{R}^{3})$ such that 
   \begin{eqnarray}\label{schwachV}
    &&\langle r^{j},\varphi\rangle_{L^{2}(\Omega,\mathbb{R}^{3})}\nonumber \\
    &=& \langle r^{j-1},\varphi\rangle_{L^{2}(\Omega,\mathbb{R}^{3})} + \frac{k}{\rho}\langle\nabla\cdot\nabla_{Y}C(x,\theta J u^{j} + (1-\theta)J u^{j-1}),\varphi\rangle_{L^{2}(\Omega,\mathbb{R}^{3})}\\ 
    && + \frac{k\theta}{\rho}\langle f^{j},\varphi\rangle_{L^{2}(\Omega,\mathbb{R}^{3})} + \frac{k(1-\theta)}{\rho}\langle f^{j-1},\varphi\rangle_{L^{2}(\Omega,\mathbb{R}^{3})}\nonumber
   \end{eqnarray}
   for all $\varphi\in H_{0}^{1}(\Omega,\mathbb{R}^{3})$.\\
   }
   \vspace{4mm}
   \\
   For the discretization in space of the weak formulation of the time discretized problem we use the Finite Element method. In particular for the implementation we used the C++ finite element library deal.II, see \cite{BangerthHartmannKanschat2007}. 
   Let be $\mathcal{V}_{h}$ a finite dimensional $H_{0}^{1}(\Omega,\RR^3)$ conform finite element space with nodal basis $\{\varphi_{1},...,\varphi_{L}\}$ with $\mbox{dim}\mathcal{V}_{h}=L<\infty$. Then we can expand all functions in the weak formulation given above in terms of the nodal basis. Hence we want to denote by a capital letter the vector of the coefficients of a function. That means for 
   example $u^{j} = \sum_{r=1}^{L}U_{r}^{j}\varphi_{r}\in H_{0}^{1}(\Omega,\mathbb{R}^{3})$ mit $U^{j}\in\mathbb{R}^{L}$. Then we get after some reformulations the following matrix equations at each time step,
  \begin{eqnarray}\label{system}
     \begin{cases}
      F'_{h}(U^{j,l})\partial U^{j,l} = - F_{h}(U^{j,l})\\
      U^{j,l+1} = U^{j,l} + \partial U^{j,l},\; U^{j,0}=U^{j-1} \\
      MR^{j} = MR^{j-1} - \frac{k}{\rho}D(u^{j},u^{j-1}) + \frac{k\theta}{\rho}MF^{j} + \frac{k(1-\theta)}{\rho}MF^{j-1}.
     \end{cases}
    \end{eqnarray}
    with the mass matrix 
    \[ M = (\langle \varphi_{r},\varphi_{s}\rangle_{H})_{r,s=1,...,L}\in\mathbb{R}^{L\times L},\]
    \begin{small}
    \begin{eqnarray}\label{FForm}
     F_{h}(U^{j,l}) 
     &=& M\big(U^{j,l}-U^{j-1}-kR^{j-1}\big) + \frac{k^{2}\theta}{\rho}D(u_{l}^{j},u^{j-1}) - \frac{k^{2}\theta^{2}}{\rho}MF^{j} \\
     && - \frac{k^{2}(1-\theta)\theta}{\rho}MF^{j-1}.\nonumber
    \end{eqnarray}
    \end{small}%
    and
    \begin{eqnarray}\label{FStrichForm}
     F'_{h}(U^{j,l})\partial U^{j,l} = \big(M + \frac{k^{2}\theta^{2}}{\rho}A(u_{l}^{j},u^{j-1})\big)\partial U^{j,l}.
    \end{eqnarray}
    In addition we used the matrix $A(u_{l}^{j},u^{j-1})\in\RR^{L\times L}$ with the entries
    \begin{eqnarray*}
    A_{rs}(u_{l}^{j},u^{j-1}) = \int\limits_{\Omega}(\nabla_{Y}\nabla_{Y}C(x,\theta J u_{l}^{j}+(1-\theta)J u^{j-1}):J\varphi_{r}):J\varphi_{s}\, dx 
   \end{eqnarray*}
   for all $r,s=1,...,L$ and the vector $D(u_{l}^{j},u^{j-1})\in\RR^{L}$ with 
   \begin{eqnarray*}
     D_{s}(u_{l}^{j},u^{j-1}) = \int\limits_{\Omega}\nabla_{Y}C(x,\theta J u_{l}^{j}+(1-\theta)J u^{j-1}):J\varphi_{s}\, dx. 
    \end{eqnarray*}
    for all $s=1,...,L$. \\
    Finally we get the following algorithm for solving the forward problem using the Conjugate Gradient (CG) method for solving the first equation in (\ref{FStrichForm}). \\
    \begin{algorithm}\label{vorw1} ($u=\mbox{forw}(\alpha)$)\\ 
  Input: $\alpha\in\mathbb{R}^{(n+1)\times(n+1)}$
  \begin{enumerate}
   \item Set $U^{0}=0\in\mathbb{R}^{L}$.
   \item Set $MR^{0}=0\in\mathbb{R}^{L}$.
   \item For every $j=1,...,m$ do
         \begin{enumerate}
         \item[3.1] Set $l=0$.
         \item[3.2] Set $U^{j,0}=U^{j-1}$.
         \item[3.3] Compute $F_{h}(U^{j,0}) = -kMR^{j-1} + \frac{k^{2}\theta}{\rho}D(u_{0}^{j},u^{j-1}) - \frac{k^{2}\theta^{2}}{\rho}MF^{j} - \frac{k^{2}(1-\theta)\theta}{\rho}MF^{j-1}$.
         \item[3.4] Compute $F'_{h}(U^{j,0}) = M + \frac{k^{2}\theta^{2}}{\rho}A(u_{0}^{j},u^{j-1})$.
         \item[3.5] Compute $\partial U^{j,0}$ with $F'_{h}(U^{j,0})\partial U^{j,0} = - F_{h}(U^{j,0})$.
         \item[3.6] Set $U^{j,1}=U^{j,0}+\partial U^{j,0}$.
         \item[3.7] While ($\|F_{h}(U^{j,l})\|>\mbox{tol}$) do
               \begin{enumerate}
               \item Set $l=l+1$.
               \item Compute \begin{small}
                              \begin{eqnarray*} 
                                F_{h}(U^{j,l})  
                                &=&M\big(U^{j,l}-U^{j-1}-kR^{j-1}\big)+\frac{k^{2}\theta}{\rho}D(u_{l}^{j},u^{j-1})-\frac{k^{2}\theta^{2}}{\rho}MF^{j} \\
                                &&-\frac{k^{2}(1-\theta)\theta}{\rho}MF^{j-1}.
                              \end{eqnarray*} 
                              \end{small}%
               \item Compute $F'_{h}(U^{j,l}) = M + \frac{k^{2}\theta^{2}}{\rho}A(u_{l}^{j},u^{j-1})$. 
               \item Compute $\partial U^{j,l}$ with $F'_{h}(U^{j,l})\partial U^{j,l} = - F_{h}(U^{j,l})$.
               \item Set $U^{j,l+1}=U^{j,l}+\partial U^{j,l}$.
               \end{enumerate}
         \item[3.8] Set $MR^{j} = MR^{j-1} - \frac{k}{\rho}D(u^{j},u^{j-1}) + \frac{k\theta}{\rho}MF^{j} + \frac{k(1-\theta)}{\rho}MF^{j-1}$.
         \end{enumerate}
  \end{enumerate}
  Output: $U^{j}\in\mathbb{R}^{L}$ for $j=0,...,m$. 
  \end{algorithm}
  \vspace{4mm}
  \\
  We proceed in the same way to develop an algorithm to solve the adjoint problem. The advantage is that this problem is linear. At first we split the corresponding differential equation with $q(t,x):=\dot{p}(t,x)$ for all 
  $(t,x)\in[0,T]\times\Omega$, too. Then we get the equivalent system 
  \begin{eqnarray}\label{splitbackward}
 \begin{cases}
    \dot{p}(t,x) - q(t,x) = 0 \\
    \rho\dot{q}(t,x) - \nabla\cdot[\nabla_{Y}\nabla_{Y}C(x,J u(t,x)):J p(t,x)] = w(t,x)
    \end{cases}
 \end{eqnarray}
 for all $(t,x)\in[0,T]\times\Omega$. After time discretization by the $\theta$-method separating $[0,T]$ in $m>0$ equal time steps with fix length $k=T/m$ and time points $t_{i}=ik$, $i=0,...,m$, we have in time step $j$ for all $j=0,...,m-1$
 \begin{eqnarray*}
   \begin{cases}
    \frac{p^{j}-p^{j+1}}{k} = \theta q^{j} + (1-\theta)q^{j+1}\\
    \rho\frac{q^{j}-q^{j+1}}{k} = \nabla\cdot[\nabla_{Y}\nabla_{Y}C(x,\theta J u^{j} + (1-\theta)J u^{j+1}):(\theta J p^{j} + (1-\theta)J p^{j+1})] \\
    \;\;\;\;\;\;\;\;\;\;\;\;\;\;\;\;\;\;\; + \theta w^{j} + (1-\theta)w^{j+1}.
   \end{cases}
  \end{eqnarray*}
  and with some reformulations
  \[
   \begin{cases}
   p^{j} = p^{j+1} + kq^{j+1} + \frac{k^{2}\theta^{2}}{\rho}\nabla\cdot[\nabla_{Y}\nabla_{Y}C(x,\theta J u^{j} + (1-\theta)J u^{j+1}):J p^{j}]\\
   \;\;\;\;\;\;\;\; + \frac{k^{2}(1-\theta)\theta}{\rho}\nabla\cdot[\nabla_{Y}\nabla_{Y}C(x,\theta J u^{j} + (1-\theta)J u^{j+1}):J p^{j+1}]\\
   \;\;\;\;\;\;\;\; + \frac{k^{2}\theta^{2}}{\rho}w^{j} + \frac{k^{2}(1-\theta)\theta}{\rho}w^{j+1}\\
   q^{j} = q^{j+1} + \frac{k\theta}{\rho}\nabla\cdot[\nabla_{Y}\nabla_{Y}C(x,\theta J u^{j} + (1-\theta)J u^{j+1}):J p^{j}]\\
   \;\;\;\;\;\;\;\; + \frac{k(1-\theta)}{\rho}\nabla\cdot[\nabla_{Y}\nabla_{Y}C(x,\theta J u^{j} + (1-\theta)J u^{j+1}):J p^{j+1}]+ \frac{k\theta}{\rho}w^{j} + \frac{k(1-\theta)}{\rho}w^{j+1}.
   \end{cases}
  \]
  Because both equations of this system are linear we can directly discretize the equations in space and after that solve the system. For this purpose we use the same notations as before. Additionally let be 
  $w\in \mathcal{V}'_{h}\subset V'$. Then we get the dual basis $(\varphi '_{r})_{r=1,...,L}$ to $(\varphi_{r})_{r=1,...,L}$, where $\langle\varphi '_{r},\varphi_{s}\rangle=\delta_{rs}$ is
   for all $r,s=1,...,L$. We obtain for the adjoint problem the system 
   \begin{eqnarray}\label{systemadj}
     \begin{cases}
      MP^{j} = MP^{j+1} + kMQ^{j+1} - \frac{k^{2}\theta^{2}}{\rho}A(u^{j},u^{j+1})P^{j} - \frac{k^{2}(1-\theta)\theta}{\rho}A(u^{j},u^{j+1})P^{j+1}\\
      \;\;\;\;\;\;\;\;\;\;\;\;\;+\frac{k^{2}\theta^{2}}{\rho}W^{j} + \frac{k^{2}(1-\theta)\theta}{\rho}W^{j+1}\\
      MQ^{j} = MQ^{j+1} - \frac{k\theta}{\rho}A(u^{j},u^{j+1})P^{j} - \frac{k(1-\theta)}{\rho}A(u^{j},u^{j+1})P^{j+1}\\
      \;\;\;\;\;\;\;\;\;\;\;\;\;+ \frac{k\theta}{\rho}W^{j} + \frac{k(1-\theta)}{\rho}W^{j+1}.
     \end{cases}
    \end{eqnarray}
    In the end with some reformulations and using the definitions 
    \begin{eqnarray}\label{S0}
   S_{0} := M + \frac{k^{2}\theta^{2}}{\rho}A(u^{j},u^{j+1})\in\mathbb{R}^{L\times L}
  \end{eqnarray}
  and 
  \begin{eqnarray}\label{S1}
   S_{1} := M-\frac{k^{2}(1-\theta)\theta}{\rho}A(u^{j},u^{j+1})\in\mathbb{R}^{L\times L}
  \end{eqnarray}
  we can reformulate (\ref{systemadj}) as
  \begin{eqnarray}\label{systemadj3}
     \begin{cases}
      S_{0}P^{j} = S_{1}P^{j+1}+ kMQ^{j+1}+\frac{k^{2}\theta^{2}}{\rho}W^{j} + \frac{k^{2}(1-\theta)\theta}{\rho}W^{j+1} \\
      MQ^{j} = MQ^{j+1} - \frac{k\theta}{\rho}A(u^{j},u^{j+1})P^{j} - \frac{k(1-\theta)}{\rho}A(u^{j},u^{j+1})P^{j+1}\\
      \;\;\;\;\;\;\;\;\;\;\;\;\;+ \frac{k\theta}{\rho}W^{j} + \frac{k(1-\theta)}{\rho}W^{j+1}.
     \end{cases}
    \end{eqnarray}
    Using again the CG method for solving the first equation of (\ref{systemadj}) we preserve the following algorithm for solving the adjoint problem:\\
  \begin{algorithm} \label{adj1} ($p=\mbox{adj1}(w,u)$) \\
  Input: $W^{j}\in\mathbb{R}^{L}$ and $u^{j}\in H_{0}^{1}(\Omega,\mathbb{R}^{3})$ for $j=0,...,m$
  \begin{enumerate}
   \item Set $P^{m}=0\in\mathbb{R}^{L}$.
   \item Set $MQ^{m}=0\in\mathbb{R}^{L}$.
   \item For every $j=m-1,...,0$ do
         \begin{enumerate}
         \item[3.1] Compute $A(u^{j},u^{j+1})$.
         \item[3.2] Compute $S_{0} = M + \frac{k^{2}\theta^{2}}{\rho}A(u^{j},u^{j+1})$.
         \item[3.3] Compute $S_{1} = M-\frac{k^{2}(1-\theta)\theta}{\rho}A(u^{j},u^{j+1})$.
         \item[3.4] Compute $P^{j}$ with \[S_{0}P^{j} = S_{1}P^{j+1}+ kMQ^{j+1}+\frac{k^{2}\theta^{2}}{\rho}W^{j} + \frac{k^{2}(1-\theta)\theta}{\rho}W^{j+1}.\]
         \item[3.5] Compute $MQ^{j}$ with \small{\begin{eqnarray*}
                                           MQ^{j} 
                                           &=& MQ^{j+1} - \frac{k\theta}{\rho}A(u^{j},u^{j+1})P^{j} - \frac{k(1-\theta)}{\rho}A(u^{j},u^{j+1})P^{j+1}\\
                                           &&+ \frac{k\theta}{\rho}W^{j} + \frac{k(1-\theta)}{\rho}W^{j+1}.
                                          \end{eqnarray*}}
         \end{enumerate}
  \end{enumerate}
  Output: $P^{j}\in\mathbb{R}^{L}$ for $j=0,...,m$. 
  \end{algorithm}
  \vspace{4mm} 
  \\

  Finally we discretize the observation operator in the same way as in \cite{BINDER;SCHUSTER:15}. 
 Let be $v\in \mathcal{V}_{h}$. Then we can write $v=\sum_{r=1}^{L}{V}_{r}\varphi_{r}$ with the basis $(\varphi_{r})_{r=1,...,L}$ of $\mathcal{V}_{h}$. It follows
 \begin{eqnarray}\label{BeobachtungDisk}
  \mathcal{Q}[v](t) = \Big(\sum\limits_{r=1}^{M}{V}_{r}(t)\int\limits_{\partial\Omega}\langle g_{k},\varphi_{r}\rangle_{\mathbb{R}^{3}}\,d\xi\Big)_{k=1,...,l} = (G({V}_{r}))_{r=1,...,L}
 \end{eqnarray}
 with $G\in\mathbb{R}^{l\times L}$ the matrix, which represents $\mathcal{Q}|_{\mathcal{V}_{h}}$ in the bases $(e_{j})_{j=1,...,l}$ of $\mathbb{R}^{l}$ and $(\varphi_{r})_{r=1,...,L}$ of $\mathcal{V}_{h}$. 
 In addition we can identify $\mathbb{R}^{l}$ respectively $L^{2}(0,T;\mathbb{R}^{l})$ with its dual space, if we choose the standard scalar product in $\mathbb{R}^{l}$. 
 Then the matrix $G^{\top}$ is exactly the representation of $Q^{*}:\,\mathbb{R}^{l}\rightarrow \mathcal{V}'_{h}$ in the bases $(e_{j})_{j=1,...,l}$ of $\mathbb{R}^{l}$ and $(\varphi'_{r})_{r=1,...,L}$ of $\mathcal{V}'_{h}$. \\
 If we choose the weight functions $g_{k}$, $k=1,...,l$, also from $\mathcal{V}_{h}$, which means $g_{k}=\sum_{r=1}^{L}G^{k}_{r}\varphi_{r}$ for all $k=1,...,l$, then we obtain by (\ref{BeobachtungDisk})
 \begin{eqnarray*}
   \mathcal{Q}[v](t) &=& \Big(\sum\limits_{r=1}^{L}{V}_{r}(t)\int\limits_{\partial\Omega}\langle g_{k},\varphi_{r}\rangle_{\mathbb{R}^{3}}\,d\xi\Big)_{k=1,...,l} \\
   &=& \Big(\sum\limits_{r=1}^{L}\sum\limits_{s=1}^{L}{V}_{r}(t)G^{k}_{s}\int\limits_{\partial\Omega}\langle \varphi_{s},\varphi_{r}\rangle_{\mathbb{R}^{3}}\,d\xi\Big)_{k=1,...,l} \\
   &=& \bar{G}M_{\partial\Omega}({V}_{r})_{r=1,...,L},
 \end{eqnarray*}
 where $M_{\partial\Omega}\in\mathbb{R}^{L\times L}$ with $(M_{\partial\Omega})_{rs}=\int_{\partial\Omega}\langle\varphi_{r},\varphi_{s}\rangle_{\mathbb{R}^3}\, d\xi$ for all $r,s=1,...,L$ is the boundary mass matrix and $\bar{G}\in\mathbb{R}^{l\times L}$ the coefficient matrix of the sensors with $\bar{G}_{ks} = G^{k}_{s}$ for all $k=1,...,l$ and $s=1,...,L$. 
 Then we get
 \begin{eqnarray*} 
  G = \bar{G}M_{\partial\Omega}\;\;\mbox{ and }\;\; G^{\top} = M_{\partial\Omega}\bar{G}^{\top}.
 \end{eqnarray*}
 With that discretization of the observation operator we are able to present the algorithm for the adjoint problem in the case (IP II) of incomplete data. We get \\
 \begin{algorithm} \label{adj2} ($p=\mbox{adj2}(w,u)$) \\
  Inpute: $W^{j}\in\mathbb{R}^{l}$ and $u^{j}\in H_{0}^{1}(\Omega,\mathbb{R}^{3})$ for $j=0,...,m$
  \begin{enumerate}
   \item Set $P^{m}=0\in\mathbb{R}^{L}$.
   \item Set $MQ^{m}=0\in\mathbb{R}^{L}$.
   \item For every $j=m-1,...,0$ do
         \begin{enumerate}
         \item[3.1] Compute $A(u^{j},u^{j+1})$.
         \item[3.2] Compute $S_{0} = M + \frac{k^{2}\theta^{2}}{\rho}A(u^{j},u^{j+1})$.
         \item[3.3] Compute $S_{1} = M-\frac{k^{2}(1-\theta)\theta}{\rho}A(u^{j},u^{j+1})$.
         \item[3.4] Compute $P^{j}$ with \[S_{0}P^{j} = S_{1}P^{j+1}+ kMQ^{j+1}+\frac{k^{2}\theta^{2}}{\rho}M_{\partial\Omega}\bar{G}^{\top}W^{j} + \frac{k^{2}(1-\theta)\theta}{\rho}M_{\partial\Omega}\bar{G}^{\top}W^{j+1}.\]
         \item[3.5] Compute $MQ^{j}$ with \small{\begin{eqnarray*}
                                           MQ^{j} 
                                           &=& MQ^{j+1} - \frac{k\theta}{\rho}A(u^{j},u^{j+1})P^{j} - \frac{k(1-\theta)}{\rho}A(u^{j},u^{j+1})P^{j+1}\\
                                           &&+ \frac{k\theta}{\rho}M_{\partial\Omega}\bar{G}^{\top}W^{j} + \frac{k(1-\theta)}{\rho}M_{\partial\Omega}\bar{G}^{\top}W^{j+1}.
                                          \end{eqnarray*}}
         \end{enumerate}
  \end{enumerate}
  Output: $P^{j}\in\mathbb{R}^{L}$ for $j=0,...,m$. 
  \end{algorithm}
  \vspace{4mm}
  \\
  We have all ingredients to formulate the algorithm for solving the inverse problem (IP II).\\
  \begin{algorithm} \label{inverse} ($\alpha=\mbox{inv}(y)$) \\
  Input: $y^{j}$ for $j=0,...,m$, basis functions $\varphi_{r}$ with $r=1,...,L$.
  \begin{enumerate}
   \item Compute the matrices $M$, $M_{\partial\Omega}$ and $\bar{G}$ with         
          \[M_{rs} = \langle\varphi_{r},\varphi_{s}\rangle_{H},\;\forall r,s=1,...,L \]
          \[(M_{\partial\Omega})_{rs} = \langle\varphi_{r},\varphi_{s}\rangle_{L^{2}(\partial\Omega,\mathbb{R}^{3})},\;\forall r,s=1,...,L \]
         and 
         \[\bar{G}_{ks} = G^{k}_{s},\;\forall k=1,...,l, s=1,...,L.\]         
   \item Set $\delta=1$.
   \item Set $\alpha = (\alpha_{rs})_{r,s=0,...,n}=1$.
   \item Set $i=0$
   \item While ($i<\mbox{maxiter}$) and ($\delta>\mbox{tol}$) do
               \begin{enumerate}
               \item[5.1] Set $i=i+1$.
               \item[5.2] Compute $u=\mbox{vorw}(\alpha)$.
               \item[5.3] Compute $\bar{u}=\bar{G}M_{\partial\Omega}u$.
               \item[5.4] Set $w=\bar{u}-y$.
               \item[5.5] Set $\delta = \frac{T}{m}\Big[\frac{1}{2}\|w^{0}\|_{\mathbb{R}^{l}}^{2}+\sum\limits_{j=1}^{m-1}\|w^{j}\|_{\mathbb{R}^{l}}^{2}+\frac{1}{2}\|w^{m}\|_{\mathbb{R}^{l}}^{2}\Big]$.
               \item[5.6] Compute $p=\mbox{adj2}(w,u)$.
               \item[5.7] For all $r,s=0,...,n$: 
                          \begin{enumerate}
                             \item[5.7.1] Compute $z^{j}=\mbox{compadj}(r,s,u^{j},p^{j})$ $\forall j=0,...,m$.
                             \item[5.7.2] Set $\gamma = \frac{T}{m}\Big[\frac{1}{2}z^{0}+\sum\limits_{j=1}^{m-1}z^{j}+\frac{1}{2}z^{m}\Big]$.
                             \item[5.7.3] Compute $\alpha_{rs} = \alpha_{rs} + \omega\gamma$.
                          \end{enumerate}
              \end{enumerate}
  \end{enumerate}
  Output: $\alpha=(\alpha_{rs})_{r,s=0,...,n}\in\mathbb{R}^{(n+1)\times(n+1)}$. 
  \end{algorithm}
  \vspace{4mm}
  \\
  Because the modifications of the algorithm for solving (IP I) are dispensable we omit to write the algorithm separately. 

\section{Convergence result}\label{sec:convergence}
In this section we prove local convergence of the attenuated Landweber iteration (\ref{Landweberalpha}) applied to (\ref{allgGlg}). 
Therefore let $u,\bar{u}\in\mathcal{A}(M_{0},M_{1},M_{2},M_{3},M_{4})$ be two solutions to the initial boundary value problem (\ref{cauchy-hyper-comb}), (\ref{anfangswert1})--(\ref{randwert}) corresponding to the parameters, initial values and right-hand sides $(\alpha, u_{0}, u_{1}, f)$ respectively $(\bar{\alpha}, u_{0}, u_{1}, f)$.
First we prove the local tangential cone condition for the parameter-to-solution map associated to the considered identification problem, because it is necessary for our proof of convergence of the attenuated Landweber method. For the proof of the cone condition we need a technical result.\\

\begin{lemma}\label{dNorm}
For $d=u(\alpha)-u(\bar{\alpha})-\mathcal{T}'(\alpha)h$ with $h=\alpha-\bar{\alpha}\in\RR^N$ under the assumption $\|h\|_{\infty}\leq r$ for a sufficiently small $r>0$ there is
\[d\in L^{2}(0,T;U)\cap H^{1}(0,T;H).\]
Thereby we define
\[\|z\|_{L^{2}(0,T;U)\cap H^{1}(0,T;H)}:=\|z\|_{L^{2}(0,T;U)}+\|z\|_{H^{1}(0,T;H)}\]
for all $z\in L^{2}(0,T;U)\cap H^{1}(0,T;H)$.\\[1ex]
\end{lemma}
\begin{proof}
Let be $u=u(\alpha)$ and $\bar{u}=u(\bar{\alpha})$. \\
With the definition of the norm of $L^{2}(0,T;U)\cap H^{1}(0,T;H)$ we get at first
\[\|d\|_{L^{2}(0,T;U)\cap H^{1}(0,T;H)}\leq \|d\|_{L^{2}(0,T;U)}+\|d\|_{H^{1}(0,T;H)}.\]
We estimate the two summands on the right hand side separately. \\
For the first one we obtain with (the proof of) Theorem \ref{glm-konv-ungl} (see \cite{SEYDEL;SCHUSTER:16})  
\[\|d\|_{L^{2}(0,T;U)}\leq\bar{L}_{2}\|h\|_{\infty}^{\frac{3}{2}}\leq\bar{L}_{2}r^{\frac{3}{2}}<\infty\]
with 
\begin{eqnarray}\label{L2bar}
  \bar{L}_{2} 
  :=\sqrt{2}\bigg(\frac{2(B_{1}+B_{3})}{\kappa(\alpha)}\exp(bT)T + \frac{4(B_{2}(h)+B_{4}\|h\|_{\infty}^{\frac{1}{2}})^{2}}{b^{2}\kappa(\alpha)^{2}}\bigg(\exp\bigg(\frac{1}{2}bT\bigg)-1\bigg)^{2}T\bigg)^{\frac{1}{2}},  
  \end{eqnarray}
  where $\bar{L}_{2}\in(0,\infty)$ holds true. \\
For the second summand we consider
\begin{eqnarray*}
 &&\|d\|_{H^{1}(0,T;H)}^{2}\\
 &=&\int\limits_{0}^{T}\|d(t,\cdot)\|_{H}^{2}+\|\dot{d}(t,\cdot)\|_{H}^{2}\, dt\\
 &=&\|d\|_{L^{2}(0,T;H)}^{2}+\|\dot{d}\|_{L^{2}(0,T;H)}^{2}.
\end{eqnarray*}
Using (\ref{L2H1}) and again Theorem \ref{glm-konv-ungl} yields
\[\|d\|_{L^{2}(0,T;H)}^{2}\leq\|d\|_{L^{2}(0,T;V)}^{2}\leq (1+C_{\Omega})\bar{L}_{2}^{2}\|h\|_{\infty}^{3}\leq(1+C_{\Omega})\bar{L}_{2}^{2}r^{3}<\infty.\]
In addition the proof of Theorem \ref{thm-glm-konv} in \cite{SEYDEL;SCHUSTER:16} delivers for $\tau\in[0,T]$ with
 \[a := 2(B_{1}+B_{3})\|h\|_{\infty}^{3},\]
 \[b := \frac{729\mu(\alpha)}{8\kappa(\alpha)}\eta M_{1}>0\]
 and
 \[k:= 2(B_{2}(h)+B_{4}\|h\|_{\infty}^{\frac{1}{2}})\|h\|_{\infty}^{\frac{3}{2}}\frac{1}{\sqrt{\kappa(\alpha)}}\]
 the estimate
 \begin{eqnarray*}
  \|\dot{d}(\tau,\cdot)\|_{L^{2}(\Omega,\RR^3)}
  \leq \exp\bigg(\frac{1}{2}b\tau\bigg)a^{\frac{1}{2}} + \frac{k}{b}\bigg(\exp\bigg(\frac{1}{2}b\tau\bigg)-1\bigg).
 \end{eqnarray*}
 In the same way as in the proof of Theorem \ref{thm-glm-konv} (see \cite{SEYDEL;SCHUSTER:16}) we get 
 \[\|\dot{d}\|_{L^{2}(0,T;H)}^{2}\leq \kappa(\alpha)\bar{L}_{2}^{2}\|h\|_{\infty}^{3}\leq\kappa(\alpha)\bar{L}_{2}^{2}r^{3}<\infty.\]
 Hence we obtain for the second summand
 \[\|d\|_{H^{1}(0,T;H)}<\infty\]
 and finally the assertion of the lemma.\hfill
\end{proof}
\\[1ex]

Now we can state the local tangential cone condition for our identification problem.\\

\begin{theorem}\label{kegelbed}
For $d:=u(\alpha)-u(\bar{\alpha})-\mathcal{T}'(\alpha)h$ with $h=\alpha-\bar{\alpha}\in\RR^N$ and a constant $L_{3}>0$ there is
\begin{eqnarray}\label{kegelUngl}
\|d\|_{L^{2}(0,T;U)\cap H^{1}(0,T;H)} \leq L_{3}\|u(\alpha)-u(\bar{\alpha})\|_{L^{2}(0,T;U)\cap H^{1}(0,T;H)}.\,
\end{eqnarray}
\end{theorem}
\begin{remark}
 We note that the inequality (\ref{kegelUngl}) is strictly speaking only a tangential cone condition if the constant $L_{3}>0$ is bounded as $L_{3}<1/2$ (see \cite{Hanke1995}). 
 We prove this under certain constraints on $\alpha,\bar{\alpha}\in\RR^{N}$ in the proof of Theorem \ref{konvlw}.\\
\end{remark}

Before we can show Theorem \ref{kegelbed}, we have to prove the subsequent lemma.\\

\begin{lemma}\label{pinU}
Let be $d=u(\alpha)-u(\bar{\alpha})-\mathcal{T}'(\alpha)h\in {L^{2}(0,T;U)\cap H^{1}(0,T;H)}$. Then for the weak solution $p=(B^{-1})^{*}d$ of 
\begin{eqnarray}
&&\hspace*{-3mm}\label{bw-1} \rho\ddot{p}(t,x)-\nabla\cdot [\nabla_{Y}\nabla_{Y}C_{\alpha}(x,J u(t,x)):J p(t,x)]=d(t,x)\\[1ex]
&&\hspace*{-3mm}\label{bw-2} p(T,x)=\dot{p}(T,x)=0,\qquad x\in\Omega\\[1ex]
&&\hspace*{-3mm}\label{bw-3} p(t,\xi)=0,\qquad (t,\xi)\in[0,T]\times\partial\Omega 
\end{eqnarray}
it is
\[p=(B^{-1})^{*}d\in L^{2}(0,T;U).\]
\end{lemma}
\begin{proof}(Proof of Lemma \ref{pinU})
Multiplying equation (\ref{bw-1}) by $2\dot{p}$ and integrating over $\Omega$ yield
\[2\langle \rho\ddot{p}(t,\cdot),\dot{p}(t,\cdot)\rangle_{H} - 2\langle \nabla\cdot[\nabla_{Y}\nabla_{Y}C_{\alpha}(\cdot,J u(t,\cdot)):J p(t,\cdot)],\dot{p}(t,\cdot)\rangle_{H} = 2\langle d(t,\cdot),\dot{p}(t,\cdot)\rangle_{H}\]
for $t\in[0,T]$. Then we get
\[a_{1}(\tau;p,p) + \rho\|\dot{p}(\tau,\cdot)\|_{H}^{2} = \int\limits_{0}^{\tau}a_{1}'(t;p,p)\, dt + 2\int\limits_{0}^{\tau}\langle d(t,\cdot),\dot{p}(t,\cdot)\rangle_{H}\;dt\]
for $\tau\in[0,T]$ with 
\[a_{1}(\tau;v,w):=\sum_{K=1}^{N}\alpha_{K}\int\limits_{\Omega}\langle\langle\nabla_{Y}\nabla_{Y}C_{K}(x,Ju(t,x)):Jv(t,x),Jw(t,x)\rangle\rangle\, dx\]
for all $v,w\in L^{2}(0,T;U)$ and therefore (see \cite{SEYDEL;SCHUSTER:16}) 
\begin{eqnarray*}
&& \kappa(\alpha)\|J p(\tau,\cdot)\|_{L^{2}(\Omega,\RR^{3\times 3})}^{2}+\rho\|\dot{p}(\tau,\cdot)\|_{H}^{2} \\
&\leq& \frac{729}{8}\eta\mu(\alpha)M_{1}\int\limits_{0}^{\tau}\|J p(t,\cdot)\|_{L^{2}(\Omega,\RR^{3\times 3})}^{2}\; dt + 2\int\limits_{0}^{\tau}\|d(t,\cdot)\|_{H}\|\dot{p}(t,\cdot)\|_{H}\; dt\\
&\leq& \frac{729}{8}\eta\mu(\alpha)\frac{M_{1}}{\kappa(\alpha)}\int\limits_{0}^{\tau}(\kappa(\alpha)\|J p(t,\cdot)\|_{L^{2}(\Omega,\RR^{3\times 3})}^{2}+\rho\|\dot{p}(t,\cdot)\|_{H}^{2})\; dt \\
&&\;\;\;+ 2\int\limits_{0}^{\tau}\frac{1}{\sqrt{\rho}}\|d(t,\cdot)\|_{H}\{\kappa(\alpha)\|J p(t,\cdot)\|_{L^{2}(\Omega,\RR^{3\times 3})}^{2}+\rho\|\dot{p}(t,\cdot)\|_{H}^{2}\}^{\frac{1}{2}}\; dt.
\end{eqnarray*}
With Gronwall's lemma for $\psi(t)=\kappa(\alpha)\|J p(t,\cdot)\|_{L^{2}(\Omega,\RR^{3\times 3})}^{2}+\rho\|\dot{p}(t,\cdot)\|_{H}^{2}$, $a=0$, \\
$b(t)=b=729\eta\mu(\alpha)M_{1}/(8\kappa(\alpha))$, $k(t)=2\|d(t,\cdot)\|_{H}/\sqrt{\rho}$ and $p=1/2$ we get for $\tau\in[0,T]$: 
\begin{eqnarray*}
\psi(\tau) \leq \exp(b\tau)\bigg[\frac{1}{2}\int\limits_{0}^{\tau}k(t)\exp(-\frac{1}{2}bt)\, dt\bigg]^{2}
= \frac{1}{4}\exp(b\tau)\bigg[\int\limits_{0}^{\tau}\frac{2}{\sqrt{\tau}}\|d(t,\cdot)\|_{H}\exp(-\frac{1}{2}bt)\,dt\bigg]^{2}.
\end{eqnarray*}
This yields together with the H\"older equation
\begin{eqnarray*}
\psi(\tau) \leq \frac{1}{\rho}\exp(b\tau)\int\limits_{0}^{\tau}\|d(t,\cdot)\|_{H}^{2}\;dt\, \int\limits_{0}^{\tau}\exp(-bt)\,dt
\leq \frac{1}{b\rho}(\exp(b\tau)-1)\|d\|_{L^{2}(0,T;H)}^{2}.
\end{eqnarray*} 
Using (\ref{L2H1}) and Theorem \ref{thm-glm-konv} it follows 
\begin{eqnarray*}
 \psi(\tau) &\leq& \frac{1}{b\rho}(\exp(b\tau)-1)\|d\|_{L^{2}(0,T;H)}^{2} \leq \frac{1}{b\rho}(\exp(b\tau)-1)\|d\|_{L^{2}(0,T;V)}^{2} \\
 &\leq& \frac{L_{2}^{2}}{b\rho}\|h\|_{\infty}^{3}(\exp(b\tau)-1).
\end{eqnarray*}
Then we have 
\begin{eqnarray*}
 \|p(\tau,\cdot)\|_{U}^{2} = \|J p(\tau,\cdot)\|_{L^{2}(\Omega,\RR^{3\times 3})}^{2} \leq \frac{L_{2}^{2}}{b\kappa(\alpha)\rho}\|h\|_{\infty}^{3}(\exp(b\tau)-1)
\end{eqnarray*}
for $\tau\in[0,T]$. Using the fact, that
\begin{equation}\label{S-tau} S(\tau):=\exp\Big(\frac{729\mu(\alpha)}{8\kappa(\alpha)}\eta M_{1}\tau\Big)-1 \end{equation}
 is monotonically increasing for $\tau\in[0,T]$, and the mean value theorem, we derive 
\[\|p\|_{L^{2}(0,T;U)}\leq\frac{TL_{2}^{2}}{b\kappa(\alpha)\rho}\|h\|_{\infty}^{3}(\exp(bT)-1)<\infty\]
and so the assertion of the lemma.\hfill
\end{proof}
\\[1ex]

Now we want to show Theorem \ref{kegelbed}.
\begin{proof}
 Let be $u=u(\alpha)$ and $\bar{u}=u(\bar{\alpha})$ the solutions of the differential equations
 \[
  \rho\ddot{u} - \sum\limits_{K=1}^{N}\alpha_{K}\nabla\cdot\nabla_{Y}C_{K}(x,J u) = f \]
  respectively
 \[ \rho\ddot{\bar{u}} - \sum\limits_{K=1}^{N}\bar{\alpha}_{K}\nabla\cdot\nabla_{Y}C_{K}(x,J \bar{u}) = f.
 \]
 In addition there is after Section \ref{sec:setup} and accordingly (\ref{gateaux-diff}) for $v=\mathcal{T}'(\alpha)h$
 \[ Bv = \sum\limits_{K=1}^{N}h_{K}\nabla\cdot\nabla_{Y}C_{K}(x,J u).\]
 Using the definition and linearity of $B$ it follows for $d=u-\bar{u}-v$ 
 \begin{eqnarray*}
  Bd &=& Bu - B\bar{u} - Bv\\
  &=& Bu - B\bar{u} - \sum\limits_{K=1}^{N}\alpha_{K}\nabla\cdot\nabla_{Y}C_{K}(x,J u) + \sum\limits_{K=1}^{N}\bar{\alpha}_{K}\nabla\cdot\nabla_{Y}C_{K}(x,J u)\\
  &=& Bu - \sum\limits_{K=1}^{N}\alpha_{K}\nabla\cdot\nabla_{Y}C_{K}(x,J u) - B\bar{u}  + \sum\limits_{K=1}^{N}\bar{\alpha}_{K}\nabla\cdot\nabla_{Y}C_{K}(x,J \bar{u})\\
  &&\;\;\; + \sum\limits_{K=1}^{N}\bar{\alpha}_{K}\nabla\cdot[\nabla_{Y}C_{K}(x,J u)-\nabla_{Y}C_{K}(x,J \bar{u})]\\
  &=& \nabla\cdot[\nabla_{Y}\nabla_{Y}C_{\alpha}(x,J u):(J\bar{u}-J u)] + \sum\limits_{K=1}^{N}\bar{\alpha}_{K}\nabla\cdot[\nabla_{Y}C_{K}(x,J u)-\nabla_{Y}C_{K}(x,J \bar{u})].
 \end{eqnarray*}
 This yields with $Y_{r}:=rJ u + (1-r)J\bar{u}$ for $r\in[0,1]$ and 
 $Y_{rs}:=sJ u + (1-s)Y_{r}$ for $s\in[0,1]$ 
 \begin{eqnarray*}
  Bd &=& \sum\limits_{K=1}^{N}\alpha_{K}\nabla\cdot[\nabla_{Y}\nabla_{Y}C_{K}(x,J u):(J\bar{u}-J u)] \\
  && \;\;\; + \sum\limits_{K=1}^{N}\bar{\alpha}_{K}\nabla\cdot\bigg[\int\limits_{0}^{1}\nabla_{Y}\nabla_{Y}C_{K}(x,Y_{r}):(J u-J\bar{u})\; dr\bigg]\\
  &=& -  \sum\limits_{K=1}^{N}\alpha_{K}\nabla\cdot\big[\nabla_{Y}\nabla_{Y}C_{K}(x,J u):(J u-J\bar{u})\\
  &&\;\;\;\;\; -\int\limits_{0}^{1}\nabla_{Y}\nabla_{Y}C_{K}(x,Y_{r}):(J u-J\bar{u})\; dr\big]\\
  &&\;\;\; -\sum\limits_{K=1}^{N}h_{K}\nabla\cdot\bigg[\int\limits_{0}^{1}\nabla_{Y}\nabla_{Y}C_{K}(x,Y_{r}):(J u-J\bar{u})\; dr\bigg]\\
  &=& -\sum\limits_{K=1}^{N}\alpha_{K}\nabla\cdot\bigg[\int\limits_{0}^{1}\int\limits_{0}^{1}\nabla_{Y}\nabla_{Y}\nabla_{Y}C_{K}(x,Y_{rs}):(1-r)(J u-J\bar{u}):(J u-J\bar{u})\; dr\, ds\bigg]\\
  &&\;\;\;-\sum\limits_{K=1}^{N}h_{K}\nabla\cdot\bigg[\int\limits_{0}^{1}\nabla_{Y}\nabla_{Y}C_{K}(x,Y_{r}):(J u-J\bar{u})\; dr\bigg].
 \end{eqnarray*}
 Then we have 
 \begin{eqnarray}\label{BdGl}
  Bd &=& -\sum\limits_{K=1}^{N}\alpha_{K}\nabla\cdot\bigg[\int\limits_{0}^{1}\int\limits_{0}^{1}\nabla_{Y}\nabla_{Y}\nabla_{Y}C_{K}(x,Y_{rs}):(1-r)(J u-J\bar{u}):(J u-J\bar{u})\; dr\, ds\bigg] \nonumber \\
  &&\;\;\;-\sum\limits_{K=1}^{N}h_{K}\nabla\cdot\bigg[\int\limits_{0}^{1}\nabla_{Y}\nabla_{Y}C_{K}(x,Y_{r}):(J u-J\bar{u})\; dr\bigg].
 \end{eqnarray}
 With Lemma \ref{dNorm} it is $d=u-\bar{u}-\mathcal{T}'(\alpha)h\in L^{2}(0,T;U)\cap H^{1}(0,T;H)$. Now let be given an arbitrary $z\in L^{2}(0,T;U)$. Then we obtain with $\tilde{u}=u-\bar{u}$ from (\ref{BdGl}) 
 \begin{eqnarray*}
  &&\langle Bd,z\rangle_{L^{2}(0,T;H)}\\
  &\leq& \bigg|-\sum\limits_{K=1}^{N}\alpha_{K}\langle\nabla\cdot\bigg[\int\limits_{0}^{1}\int\limits_{0}^{1}\nabla_{Y}\nabla_{Y}\nabla_{Y}C_{K}(x,Y_{rs}):(1-r)(J\tilde{u}):(J\tilde{u})\; dr\, ds\bigg],z\rangle_{L^{2}(0,T;H)}\\
  &&\;\;\;-\sum\limits_{K=1}^{N}h_{K}\langle\nabla\cdot\bigg[\int\limits_{0}^{1}\nabla_{Y}\nabla_{Y}C_{K}(x,Y_{r}):(J\tilde{u})\; dr\bigg],z\rangle_{L^{2}(0,T;H)}\bigg|\\
 \end{eqnarray*}
 and with partial integration, the triangle inequality and $\alpha_{K}\geq 0$ for all $K=1,...,N$ 
 \begin{eqnarray*}
  &\leq& \sum\limits_{K=1}^{N}\alpha_{K}\bigg|\int\limits_{0}^{T}\int\limits_{\Omega}\int\limits_{0}^{1}\int\limits_{0}^{1}(\nabla_{Y}\nabla_{Y}\nabla_{Y}C_{K}(x,Y_{rs}):(1-r)J\tilde{u}:J\tilde{u}):J z\;dr\,ds\,dx\,dt\bigg|\\
  &&\;\;\; +\sum\limits_{K=1}^{N}|h_{K}|\bigg|\int\limits_{0}^{T}\int\limits_{\Omega}\int\limits_{0}^{1}(\nabla_{Y}\nabla_{Y}C_{K}(x,Y_{r}):(J\tilde{u})):J z\;dr\,dx\,dt\bigg|.
 \end{eqnarray*}
 Using (\ref{beschr1}), (\ref{bed2}), the H\"older inequality and Corollary \ref{M4} we can further estimate
 \begin{eqnarray*}
 &&\langle Bd,z\rangle_{L^{2}(0,T;H)}\\
  &\leq& \sum\limits_{K=1}^{N}\alpha_{K}\sum\limits_{i,j,k,l,p,q=1}^{3}\int\limits_{0}^{T}\int\limits_{\Omega}\int\limits_{0}^{1}\int\limits_{0}^{1}(1-r)|\partial_{Y_{ij}}\partial_{Y_{kl}}\partial_{Y_{pq}}C_{K}(x,Y_{rs})||\partial_{j}\tilde{u}_{i}|
  |\partial_{l}\tilde{u}_{k}||\partial_{q}z_{p}|\;dr\,ds\,dx\,dt \\
 &&\;\;\; +\sum\limits_{K=1}^{N}|h_{K}|\sum\limits_{i,j,k,l=1}^{3}\int\limits_{0}^{T}\int\limits_{\Omega}\int\limits_{0}^{1}|\partial_{Y_{ij}}\partial_{Y_{kl}}C_{K}(x,Y_{r})||\partial_{j}\tilde{u}_{i}||\partial_{l}z_{k}|\;dr\,dx\,dt\\
 &\leq& \frac{1}{2}\sum\limits_{K=1}^{N}\alpha_{K}\mu_{K}^{[2]}\int\limits_{0}^{T}\int\limits_{\Omega}\Big(\sum\limits_{i,j=1}^{3}|\partial_{j}\tilde{u}_{i}|\Big)^{2}\Big(\sum\limits_{p,q=1}^{3}|\partial_{q}z_{p}|\Big)\;dx\,dt\\
 &&\;\;\; +\sum\limits_{K=1}^{N}|h_{K}|\mu_{K}^{[1]}\int\limits_{0}^{T}\int\limits_{\Omega}\Big(\sum\limits_{i,j=1}^{3}|\partial_{j}\tilde{u}_{i}|\Big)\Big(\sum\limits_{k,l=1}^{3}|\partial_{l}z_{k}|\Big)\;dx\,dt\\
 &\leq& \frac{1}{2}\sum\limits_{K=1}^{N}\alpha_{K}\mu_{K}^{[2]}\int\limits_{0}^{T}\Big(\int\limits_{\Omega}\Big(\sum\limits_{i,j=1}^{3}|\partial_{j}\tilde{u}_{i}|\Big)^{4}\;dx\Big)^{\frac{1}{2}}\Big(\int\limits_{\Omega}\Big(\sum\limits_{p,q=1}^{3}|\partial_{q}z_{p}|\Big)^{2}\;dx\Big)^{\frac{1}{2}}\,dt\\
 &&\;\;\; +\sum\limits_{K=1}^{N}|h_{K}|\mu_{K}^{[1]}\int\limits_{0}^{T}\Big(\int\limits_{\Omega}\Big(\sum\limits_{i,j=1}^{3}|\partial_{j}\tilde{u}_{i}|\Big)^{2}\;dx\Big)^{\frac{1}{2}}\Big(\int\limits_{\Omega}\Big(\sum\limits_{k,l=1}^{3}|\partial_{l}z_{k}|\Big)^{2}\;dx\Big)^{\frac{1}{2}}\,dt\\
 &\leq& \Big(\frac{81}{2}M_{4}\|\alpha\|_{\infty}\sum\limits_{K=1}^{N}\mu_{K}^{[2]}+ 9\|h\|_{\infty}\sum\limits_{K=1}^{N}\mu_{K}^{[1]}\Big)\|\tilde{u}\|_{L^{2}(0,T;U)}\|z\|_{L^{2}(0,T;U)}.
 \end{eqnarray*}
 Then we have with 
 \begin{eqnarray}\label{L3}
  \bar{L}_{3}:=\frac{81}{2}M_{4}\|\alpha\|_{\infty}\sum\limits_{K=1}^{N}\mu_{K}^{[2]}+ 9\|h\|_{\infty}\sum\limits_{K=1}^{N}\mu_{K}^{[1]}>0
 \end{eqnarray}
 the inequality
 \begin{eqnarray}\label{BdzUngl}
 \langle Bd,z\rangle_{L^{2}(0,T;H)} \leq \bar{L}_{3}\|\tilde{u}\|_{L^{2}(0,T;U)}\|z\|_{L^{2}(0,T;U)}
 \end{eqnarray}
 for all $z\in L^{2}(0,T;U)$.\\ 
 Now let be $z=(B^{-1})^{*}d$. In Lemma \ref{pinU} we have proven that $z=(B^{-1})^{*}d\in L^{2}(0,T;U)$. Then it is
 \begin{eqnarray*}
 \langle Bd,z\rangle_{L^{2}(0,T;H)} &=& \int\limits_{0}^{T}\langle Bd(t,\cdot),(B^{-1})^{*}d(t,\cdot)\rangle\;dt = \langle B^{-1}Bd, d\rangle_{L^{2}(0,T;U)\cap H^{1}(0,T;H)} \\
  &=&\|d\|_{L^{2}(0,T;U)\cap H^{1}(0,T;H)}^{2}  
 \end{eqnarray*}
 and therefore
 \begin{eqnarray}\label{dGleich}
  \langle Bd,z\rangle_{L^{2}(0,T;H)} = \|d\|_{L^{2}(0,T;U)\cap H^{1}(0,T;H)}^{2}.
 \end{eqnarray}
 Furthermore we have with (\ref{poincare-1}) and following the same lines as in the proof of Lemma \ref{stetigeinb} (see \cite{SEYDEL;SCHUSTER:16})
 \begin{eqnarray*}
  \|z\|_{L^{2}(0,T;U)}^{2} &=& \int\limits_{0}^{T}\langle (B^{-1})^{*}d(t,\cdot),(B^{-1})^{*}d(t,\cdot)\rangle_{U}\;dt\\
  &\leq& \|d\|_{L^{2}(0,T;U)\cap H^{1}(0,T;H)}\|B^{-1}z\|_{L^{2}(0,T;U)\cap H^{1}(0,T;H)}\\
  &\leq& C\sqrt{1+C_{\Omega}}\|d\|_{L^{2}(0,T;U)\cap H^{1}(0,T;H)}\|z\|_{L^{2}(0,T;U)}
 \end{eqnarray*}
 and thus
 \begin{eqnarray}\label{zUngl}
  \|z\|_{L^{2}(0,T;U)} \leq C\sqrt{1+C_{\Omega}}\|d\|_{L^{2}(0,T;U)\cap H^{1}(0,T;H)}.
 \end{eqnarray}
 Finally we obtain from (\ref{BdzUngl}), (\ref{dGleich}) and (\ref{zUngl})
 \[\|d\|_{L^{2}(0,T;U)\cap H^{1}(0,T;H)}^{2} \leq \bar{L}_{3}C\sqrt{1+C_{\Omega}}\|\tilde{u}\|_{L^{2}(0,T;U)}\|d\|_{L^{2}(0,T;U)\cap H^{1}(0,T;H)}\]
 und thereby 
 \begin{eqnarray}\label{kegel}
  \|d\|_{L^{2}(0,T;U)\cap H^{1}(0,T;H)} \leq \bar{L}_{3}C\sqrt{1+C_{\Omega}}\|\tilde{u}\|_{L^{2}(0,T;U)\cap H^{1}(0,T;H)},
 \end{eqnarray}
 because $\tilde{u}=u(\alpha)-u(\bar{\alpha})=u(\alpha)-u(\bar{\alpha})-u'+u'=d+u'\in {L^{2}(0,T;U)\cap H^{1}(0,T;H)}$ follows directly from Lemma \ref{dNorm}, $u'=\mathcal{T}'(\alpha)h\in\mathcal{X}$ and Lemma \ref{stetigeinb}.
Setting $L_{3}:=\bar{L}_{3}C\sqrt{1+C_{\Omega}}>0$ this finally gives the assertion.\hfill
 \end{proof}
 \\[1ex]
 
 We prove now the main result of this section.\\
 
\begin{theorem}\label{konvlw}
  Let be $\mathcal{T}(\alpha)=\tilde{u}$ solvable in a ball $\mathcal{B}_{\rho}(\alpha^{(0)})\subset D(\mathcal{T})$ centered about the initial value $\alpha^{(0)}\in\mathbb{R}_{+}^{N}$ with 
  \begin{eqnarray*}
  \|\alpha^{(0)}\|_{\infty} < \frac{1}{81M_{4}\sum_{K=1}^{N}\mu_{K}^{[2]}}
  \end{eqnarray*}
  and radius
  \begin{eqnarray}\label{rhoabsch}
  0 < \rho < \frac{1-81C\sqrt{1+C_{\Omega}}M_{4}\|\alpha^{(0)}\|_{\infty} \sum_{K=1}^{N}\mu_{K}^{[2]}}{18C\sqrt{1+C_{\Omega}}\sum_{K=1}^{N}(9M_{4}\mu_{K}^{[2]}+4\mu_{K}^{[1]})}.
  \end{eqnarray}
  Then the attenuated Landweber iteration converges under the assumption 
  \begin{eqnarray}\label{omega}
  \omega\in\bigg(0,\frac{1}{L_{1}^{2}}\bigg)
  \end{eqnarray}
  with $L_{1}>0$ from Theorem \ref{gateaux-stetig} applied to input data $\tilde{u}\in L^{2}(0,T;U)\cap H^{1}(0,T;H)$ to a solution of $\mathcal{T}(\alpha)=\tilde{u}$. 
  If $\mathcal{N}(\mathcal{T}'(\alpha^{\dagger}))\subset\mathcal{N}(\mathcal{T}'(\alpha))$ for all $\alpha\in\mathcal{B}_{\rho}(\alpha^{\dagger})$, then $\alpha^{(k)}$ converges to the minimum-norm-solution $\alpha^{\dagger}$ as $k\rightarrow\infty$.\\
  \end{theorem}
  \begin{proof}
   At first the operator $\mathcal{T}$ is Fr\'{e}chet-differentiable because of Theorem \ref{thm-glm-konv}. Furthermore there is $\|\mathcal{T}'(\alpha)\|\leq L_{1}$ for a $L_{1}>0$ due to Theorem \ref{gateaux-stetig} and hence the Fr\'{e}chet-derivative is bounded. 
   Thereby $\mathcal{T}$ is continuously Fr\'{e}chet-differentiable. In addition we deduce again from Theorem \ref{gateaux-stetig} that $\|\mathcal{T}'(\alpha)\|\leq L_{1}$ holds true for a constant $L_{1}>0$ and for all $\alpha\in D(\mathcal{T})$ and hence also for all $\alpha\in \mathcal{B}_{\rho}(\alpha^{(0)})\subset D(\mathcal{T})$. 
   For this reason we get with (\ref{omega}) the estimation $\sqrt{\omega}\|\mathcal{T}'(\alpha)\|\leq\sqrt{\omega}L_{1}\leq 1$ for all $\alpha\in \mathcal{B}_{\rho}(\alpha^{(0)})$. \\
   According to Theorem \ref{kegelbed} we estimate  
   \begin{eqnarray*}
  \|\mathcal{T}(\tilde{\alpha})-\mathcal{T}(\alpha)-\mathcal{T}'(\alpha)(\tilde{\alpha}-\alpha)\|_{L^{2}(0,T;U)\cap H^{1}(0,T;H)}\leq\eta_{KB}\|\mathcal{T}(\tilde{\alpha})-\mathcal{T}(\alpha)\|_{L^{2}(0,T;U)\cap H^{1}(0,T;H)}
  \end{eqnarray*}  
 for $\alpha,\tilde{\alpha}\in\mathcal{B}_{2\rho}(\alpha^{(0)})$ and $\eta_{KB}=\bar{L}_{3}C\sqrt{1+C_{\Omega}}$. 
Using (\ref{rhoabsch}) we obtain 
   \begin{eqnarray*}
   \eta_{KB} &=& \bar{L}_{3}C\sqrt{1+C_{\Omega}} \\
   &=& \frac{81}{2}M_{4}C\sqrt{1+C_{\Omega}}\|\tilde{\alpha}\|_{\infty}\sum\limits_{K=1}^{N}\mu_{K}^{[2]} + 9C\sqrt{1+C_{\Omega}}\|\tilde{\alpha}-\alpha\|_{\infty}\sum\limits_{K=1}^{N}\mu_{K}^{[1]}\\
   &\leq& \frac{81}{2}M_{4}C\sqrt{1+C_{\Omega}}\|\alpha^{(0)}\|_{\infty}\sum\limits_{K=1}^{N}\mu_{K}^{[2]} + 9C\sqrt{1+C_{\Omega}}\sum\limits_{K=1}^{N}\bigg(\frac{9}{2}M_{4}\mu_{K}^{[2]}2\rho+4\rho\mu_{K}^{[1]}\bigg)\\
   &=& \frac{81}{2}M_{4}C\sqrt{1+C_{\Omega}}\|\alpha^{(0)}\|_{\infty}\sum\limits_{K=1}^{N}\mu_{K}^{[2]} + 9C\sqrt{1+C_{\Omega}}\sum\limits_{K=1}^{N}\bigg(9M_{4}\mu_{K}^{[2]}+4\mu_{K}^{[1]}\bigg)\rho \\
   &<& \frac{81}{2}M_{4}C\sqrt{1+C_{\Omega}}\|\alpha^{(0)}\|_{\infty}\sum\limits_{K=1}^{N}\mu_{K}^{[2]} + \frac{1}{2}\bigg(1-81M_{4}C\sqrt{1+C_{\Omega}}\|\alpha^{(0)}\|_{\infty}\sum\limits_{K=1}^{N}\mu_{K}^{[2]}\bigg)\\
   &=& \frac{1}{2}.
   \end{eqnarray*}
   Hence, $\eta_{KB}<1/2$. For this reason all assumptions of Theorem 11.4 of \cite{engl1996regularization} and Theorem 2.4 of \cite{kaltenbacher2008iterative}
  are satisfied.\hfill
  \end{proof}
  \\[1ex]
  
  Using the last theorem we want to show the following convergence result for noisy data \\$u^{\delta}\in L^{2}(0,T;U)\cap H^{1}(0,T;H)$ with $\|u^{\delta}-\tilde{u}\|_{L^{2}(0,T;U)\cap H^{1}(0,T;H)}\leq \delta$ for one $\delta>0$.\\
  
  \begin{corollary}\label{konvlwgest}
  Given the assumptions of Theorem \ref{konvlw} and that the Landweber iteration applied to $u^{\delta}$ is stopped with $k_{*}=k(\delta,u^{\delta})$ according to the discrepancy principle 
  \begin{eqnarray}\label{diskrepanzalpha}
  \|u^{\delta}-\mathcal{T}(\alpha^{(k_{*}),\delta})\|_{L^{2}(0,T;U)\cap H^{1}(0,T;H)}\leq \tau\delta < \|u^{\delta}-\mathcal{T}(\alpha^{(k),\delta})\|_{L^{2}(0,T;U)\cap H^{1}(0,T;H)}
  \end{eqnarray}
   for all $0\leq k<k_{*}$ with a positive tolerance parameter $\tau$ satisfying
  \begin{eqnarray}\label{tau}
  \tau > 2\frac{1+\eta_{KB}}{1-2\eta_{KB}}>2,
  \end{eqnarray}
  then $\alpha^{(k_{*}),\delta}$ converges to a solution $\alpha\in\mathcal{B}_{\rho}(\alpha^{(0)})$ of $\mathcal{T}(\alpha)=\tilde{u}$ as $\delta\rightarrow 0$. \\[1ex]
  \end{corollary}
  \begin{proof}
  In the proof of Theorem \ref{konvlw} we showed that all assumptions of Theorem 11.4 of \cite{engl1996regularization} and hence of Theorem 11.5 of \cite{engl1996regularization} are fullfilled. This yields the assertion.\hfill
  \end{proof}
  
\section{Numerical results}\label{sec:results}
To verify the ability of the method to localize damages we performed some test computations. We assume that all entries of the coefficient matrix of the undamaged plate are equal to $1$. We simulate a wave propagating through a plate including a damage which is modeled by entries of the coefficient matrix $\alpha$ different from $1$. In this way we generate the displacement field $u_{S}$. 
The input to our algorithm was then the output of the observation operator applied to $u_{S}$
\[\tilde{u}=\mathcal{Q}u_{S}.\]
We consider a plate of Neo-Hookean material with measures $6.7\mbox{mm}\times 1\mbox{m}\times 1\mbox{m}$. The plate is discretized by $4\times 30\times 30$ trilinear elements. Furthermore the simulation time interval is $[0,T]$ with $T=133\mu\mbox{s}$ subdivided equidistantly by 16 time points. 
For the numerical consideration we assume that we have a plate with the measures $[-0.1,0.1]\times[-15,15]^{2}$ with 4 cells in $x_{1}$-direction and 30 cells in $x_{2}$- and $x_{3}$-direction, such that there is an equidistant decomposition $-15=x_{l}^{(0)}<...<x_{l}^{(30)}=15$, $l=2,3$ of $[-15,15]$. In addition we use in the algorithm a time intervall $[0,4]$ again subdivided equidistantly by 16 time points.
As a next step we want to specify the material model. We use for the stored energy function $C(x,Y)$ a conic combination 
\[C(x,Y)=\sum\limits_{K=1}^{N}\alpha_{K}C_{K}(x,Y),\] 
$\alpha_{K}\geq 0$ for all $K=1,...,N$, of tensor products 
\[C_{K}(x,Y)=v_{K}(x)C(Y)\]
for all $K=1,...,N$. As noted before we consider a Neo-Hookean material, such that we have 
\begin{eqnarray}\label{NeoHook}
 C(Y) = c_{1}(I_{1}-3) + \frac{c_{1}}{\beta}(D^{-2\beta} -1).
\end{eqnarray}
with $I_{1}=\|F\|_{F}^{2}$ and $D=\det(F)$ for $F=Y+I=Ju+I$ as well as $\beta=\frac{3K-2\mu}{6\mu}>0$ and $c_{1}=\frac{\mu}{2}>0$. Therefore we set $K=68.6\mbox{GPa}$ and $\mu=26.32\mbox{GPa}$ as in the article \cite{RauLa2015}. For this function we compute 
\[\nabla_{Y}C(Y) = 2c_{1}Y - 2c_{1}D^{-2\beta}Y^{-\top}\]
and
\begin{eqnarray*}
                            \nabla_{Y}\nabla_{Y}C(Y):H = 2c_{1}H + 4c_{1}\beta D^{-2\beta}(Y^{-\top}\otimes Y^{-\top}):H + 2c_{1}D^{-2\beta}Y^{-\top}H^{\top}Y^{-\top} 
                           \end{eqnarray*}
with an arbitrary tensor $H$ of second order. For the functions $v_{K}(x)$, $K=1,...,N$, we will appropriate B-Splines. Due to the assumption that if there is a defect, then it occurs at the boundary of the plate as for example in the case of delaminations, we consider in $x_{1}$-direction two layers $I$ and $II$ near the boundary of the plate. 
For the numerical treatments we choose $I$ at $x_{1}=-0.05$ and $II$ at $x_{2}=0.05$. Corresponding to the assumption above we have only undamaged material and set $C_{K}(x,Y)=C(Y)$ for all $K=1,...,N$ between the two layers. At the two layers it is either $x_{1}=-0.05$ or $x_{1}=0.05$ and so there the stored energy function depends only on $x_{2}$ and $x_{3}$ relating to the space. 
Then we choose for the stored energy function 
\begin{eqnarray}\label{konKomb}
  C(x,Y) = \sum\limits_{i=0}^{30}\sum\limits_{j=0}^{30}\alpha_{ij}b_{i}(x_{2})b_{j}(x_{3})C(Y)
 \end{eqnarray}
 with $\alpha_{ij}\geq 0$ for all $i,j=0,...,30$ and $b_{i}(x_{2})$ and $b_{j}(x_{3})$ are B-Splines of first order. Finally, it can be proven that if we consider the B-Splines at the nodes $x_{l}^{(k)}$, $k=0,...,30$, $l=2,3$, of the equidistant decomposition of $[-15,15]$, then we have 
 \[b_{i}(x_{2}^{(p)})=\delta_{ip}\;\mbox{ and }\; b_{j}(x_{2}^{(q)})=\delta_{jq}\]
 for all $i,j,p,q=0,...,30$. That yields to 
 \[C(x_{2}^{(p)},x_{3}^{(q)},Y) = \sum\limits_{i=0}^{n}\sum\limits_{j=0}^{n}\alpha_{ij}\delta_{ip}\delta_{jq}C(Y)\]
 and the corresponding derivatives for all $p,q=0,...,30$ at the two layers.\\
In every simulation the wave is excited at the beginning of the simulation time at one point, at the center of the plate. The excitation signal was chosen as 
\begin{eqnarray}\label{Kraft}
   f(x_{1}, x_{2}, x_{3}, t) = f_{t}(t)f_{2}(x_{2})f_{3}(x_{3})\begin{pmatrix}
                                                                0\\
                                                                0\\
                                                                1
                                                               \end{pmatrix}
  \end{eqnarray}
with $f_{t}$, $f_{2}$ and $f_{3}$ defined as in figure \ref{fig:Anregung}. 
\begin{small}
   \begin{figure}[ht]
  \centering
  \includegraphics[width=0.6\textwidth]{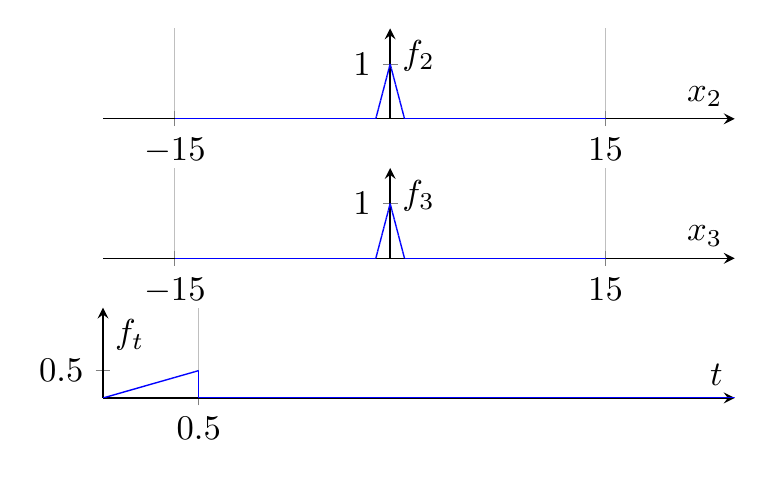}
  \caption{Plot of the factor functions of the excitation signal}
  \label{fig:Anregung}
  \end{figure}
  \end{small}%
Thus the excitation signal acts in $x_{3}$-direction. We have taken this approach for the excitation signal from \cite{BINDER;SCHUSTER:15}. The only difference is that their the wave is excited at four different locations. 
The damaged areas are of cylindrical shape with the axis in thickness direction $x_{1}$ and a square as base with side length equal to one. There were three scenarios 
\begin{itemize}
  \item One damaged area approximately in the center of the plate at $(-1.5,1.5)$ (Scenario A)
  \item Two damaged areas, one of them nearer to to center of the plate at $(5.5,5.5)$, the other nearer to one edge of the plate at $(-1.5,-10.5)$ (Scenario B)
  \item Two damaged areas, both relatively close to the center of the plate at $(-1.5,-4.5)$ and $(5.5,5.5)$ (Scenario C)
 \end{itemize}
 We also stated the center of the square representing the respective damage area. The three damage scenarios are plotted in figure \ref{fig:schadensszenarien}.
 \begin{small}
 \begin{figure}[ht]
\centering
\begin{minipage}[t]{.3\linewidth}
    \centering
\begin{tikzpicture}[every node/.style={color=black, font=\sffamily\Huge},scale=0.8]
 \draw[draw=black] (2.5,-2.5) -- (2.5,2.5) -- (-2.5,2.5) -- (-2.5,-2.5) -- cycle;
 \node[scale=6,text opacity=0.2] at (0.0,0.0) {A};
 \draw[draw=black] (-0.33,0.33) -- (-0.166,0.33) -- (-0.166,0.166) -- (-0.33,0.166)  -- cycle;
\end{tikzpicture}
\end{minipage}
\hfill
\begin{minipage}[t]{.3\linewidth}
    \centering
    \begin{tikzpicture}[every node/.style={color=black, font=\sffamily\Huge},scale=0.8]
     \draw[draw=black] (2.5,-2.5) -- (2.5,2.5) -- (-2.5,2.5) -- (-2.5,-2.5) -- cycle;
     \node[scale=6,text opacity=0.2] at (0.0,0.0) {B};
     \draw[draw=black] (0.833,-0.833) -- (1,-0.833) -- (1,-1) -- (0.833,-1)  -- cycle;
     \draw[draw=black] (-0.33,1.833) -- (-0.166,1.833) -- (-0.166,1.666) -- (-0.33,1.666)  -- cycle;
    \end{tikzpicture}
\end{minipage}
\hfill
\begin{minipage}[t]{.3\linewidth}
    \centering
    \begin{tikzpicture}[every node/.style={color=black, font=\sffamily\Huge},scale=0.8]
     \draw[draw=black] (2.5,-2.5) -- (2.5,2.5) -- (-2.5,2.5) -- (-2.5,-2.5) -- cycle;
     \node[scale=6,text opacity=0.2] at (0.0,0.0) {C};
     \draw[draw=black] (0.833,-0.833) -- (1,-0.833) -- (1,-1) -- (0.833,-1)  -- cycle;
     \draw[draw=black] (-0.33,0.833) -- (-0.166,0.833) -- (-0.166,0.666) -- (-0.33,0.666)  -- cycle;
    \end{tikzpicture}
\end{minipage}
\caption[The three damage scenarios]{The three damage scenarios. The black squares indicate the damage areas.}
\label{fig:schadensszenarien}
\end{figure}
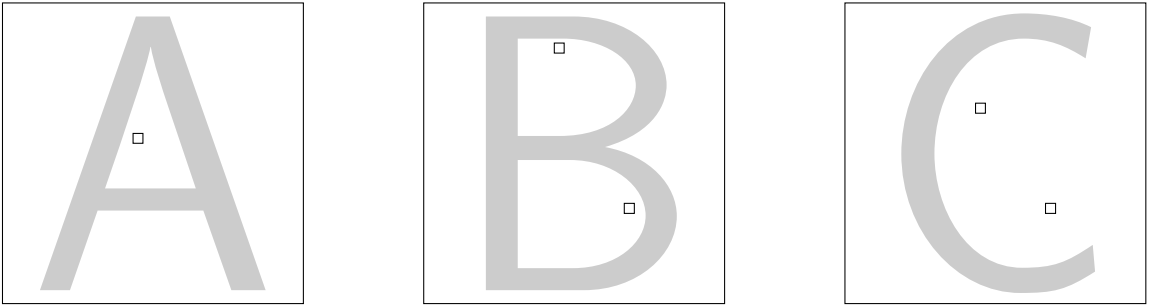
\end{small}
Let be $(\varphi_{r})_{r=1,...,L}$ the Lagrangian basis of the Galerkin space of the chosen Finite Element discretization. Then every basis function is at exactly one node equal to $1$ and at all other ones $0$. According to \cite{BINDER;SCHUSTER:15} we choose a basis element confined to the boundary of the plate as 
weighting function for every sensor of the observation operator. It follows that the observation matrix $W$ has the structure of a diagonal matrix multiplied with the boundary matrix $M_{\partial\Omega}$, where
$(M_{\partial\Omega})_{rs} = \langle\varphi_{r},\varphi_{s}\rangle_{L^{2}(\partial\Omega,\mathbb{R}^{3})}$ with $r,s=1,...,L$ if $(\varphi_{r})_{r=1,...,L}$ is the Lagrangian basis. In this way for the description of the observation operator it is sufficient to indicate the nodes, which are associated with a basis element that is at the same time 
a weighting function. \\
The relaxation parameter $\omega$ was in all examples set to $10.0$. The value of the displacement $u_{S}$ is estimated experimentally by solving the problem for a given coefficient matrix $\alpha\in\mathbb{R}^{31\times 31}$. To interpret the results we plot the coefficient matrix, where several colors represent different values of the single entries of the 
coefficient matrix.\\
For our experiments we assumed initially to have complete data. In a first experiment we have computed 50 Iterations of the attenuated Landweber method for all three damage scenarios. The results are presented in figure \ref{fig:alledrei}.
\begin{figure}[ht]
\begin{minipage}[t]{.33\linewidth}
    \centering
    \includegraphics[width=1.0\linewidth]{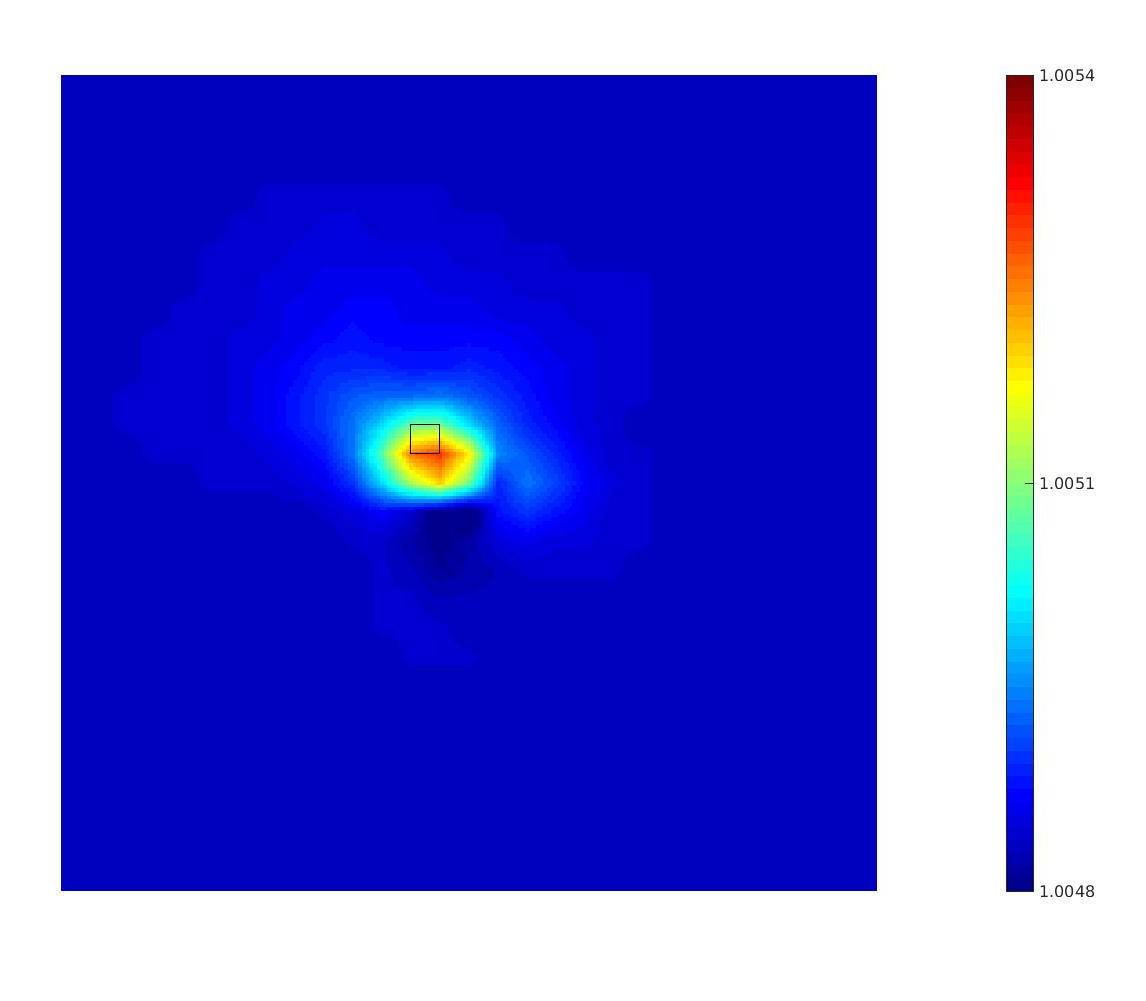}
\end{minipage}
\hspace{0.1cm}
\begin{minipage}[t]{.33\linewidth}
    \centering
    \includegraphics[width=1.0\linewidth]{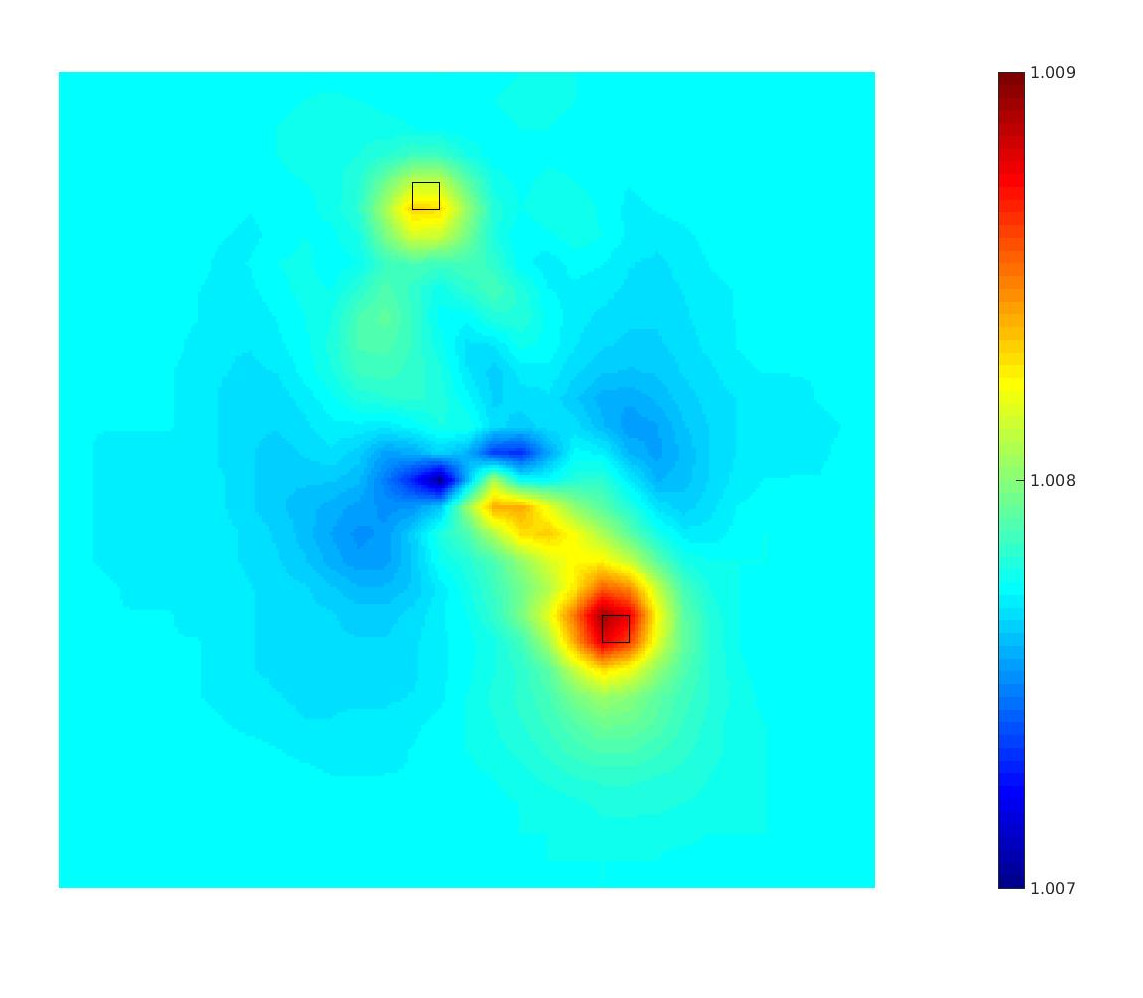}
\end{minipage}
\hspace{0.1cm}
\begin{minipage}[t]{.33\linewidth}
    \centering
    \includegraphics[width=1.0\linewidth]{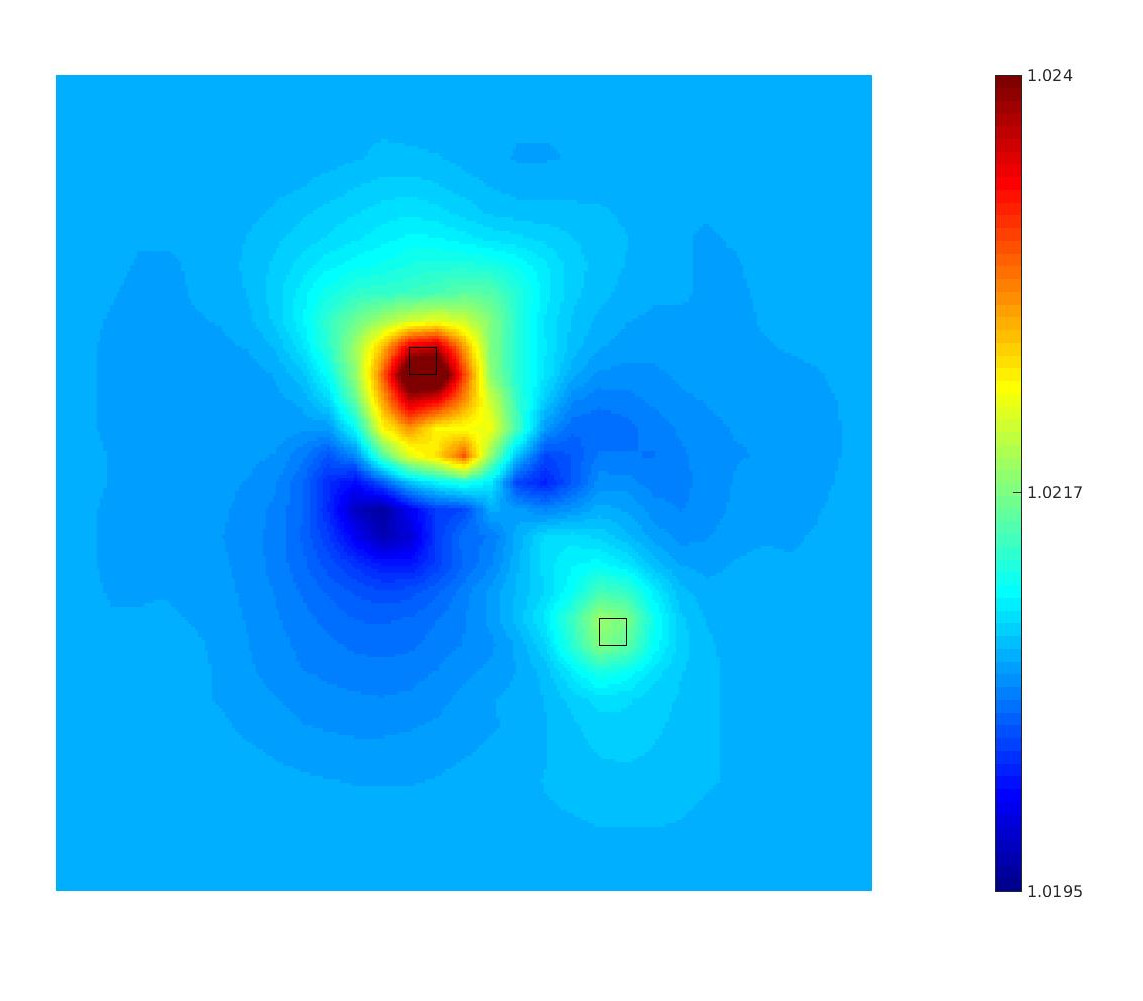}
\end{minipage}
\caption[The results of experiment 1]{The results of experiment 1 with damage scenarios A, B and C from left to right.}
\label{fig:alledrei}
\end{figure}
We see that our method can detect and localize the given damages. Regarding the damage scenarios B and C and therefore in case of two damaged areas it can be seen that there are artifacts between the center of the plate and the damaged area nearer to the center. One might suspect that the artifacts arise due to reflecions from the first damaged area meeting by the wave. 
To emphasize that we consider two damage scenarios with two damage areas. \\
In a second experiment we want to see how the results of the algorithm change in case of noisy data. For that purpose we add a random vector to the input data $\tilde{u}$ to get the noisy data $u^{\delta}$ with 
\[\frac{\|U^{\delta}-\tilde{U}\|_{L^{2}(0,T;\RR^L)}}{\|\tilde{U}\|_{L^{2}(0,T;\RR^L)}}=\delta.\]
Here we use the representations $u^{\delta}(t_{j})=\sum_{r=1}^{L}U^{\delta}_{r}(t_{j})\varphi_{r}$ and $\tilde{u}(t_{j})=\sum_{r=1}^{L}\tilde{U}^{\delta}_{r}(t_{j})\varphi_{r}$ at the time points $t_{j}=jT/m$ for all $j=0,...,m$ (see Section \ref{sec:scheme}). For our calculations we use different values up to 1 for $\delta$. 
In figure \ref{fig:rauschen} one can see the difference 
between $\tilde{u}$ and $u^{\delta}$ for $\delta=1$ at time $t=108\mu\mbox{s}$.
\begin{figure}[ht]
\centering
\begin{minipage}[t]{.27\linewidth}
    \centering
    \includegraphics[width=1.0\linewidth]{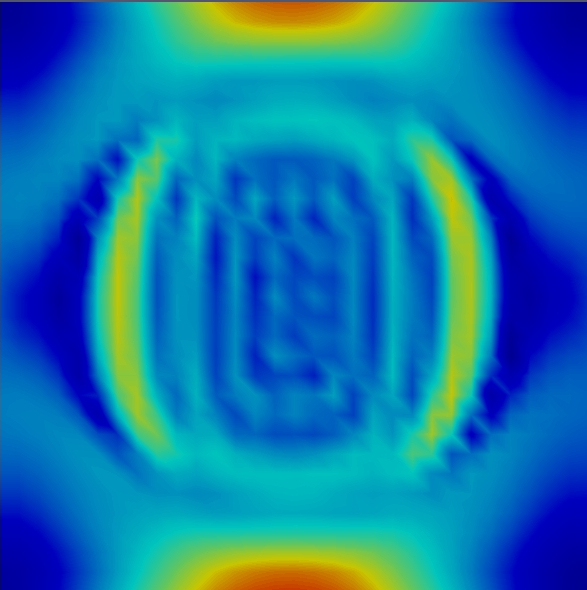}
\end{minipage}
\hspace{0.8cm}
\begin{minipage}[t]{.27\linewidth}
    \centering
    \includegraphics[width=1.0\linewidth]{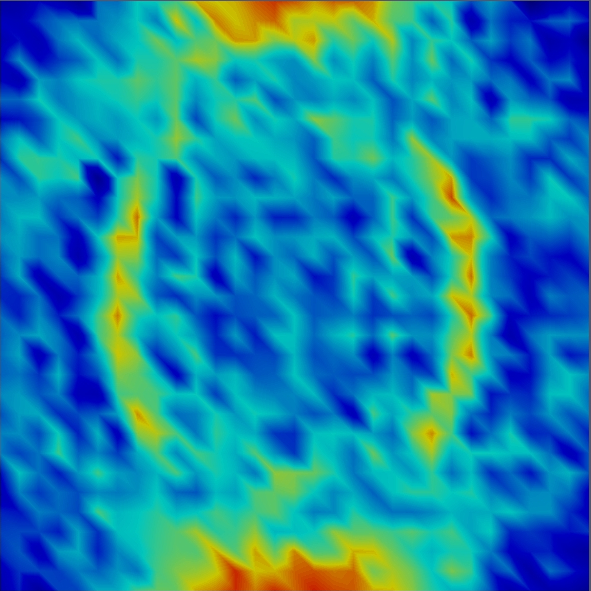}
\end{minipage}
\caption{The difference between exact and noisy input data with $\delta=1$. The displacement field corresponds to damage scenario A taken at time $t=108\mu\mbox{s}$. The coloring corresponds to the norm of the displacement field.}
\label{fig:rauschen}
\end{figure}
The result of experiment 2 is given in Figure \ref{fig:exp2} and proves the stability of the algorithm with respect to noise. \\
\begin{figure}[ht]
\centering
\begin{minipage}[t]{.4\linewidth}
    \centering
    \includegraphics[width=1.0\linewidth]{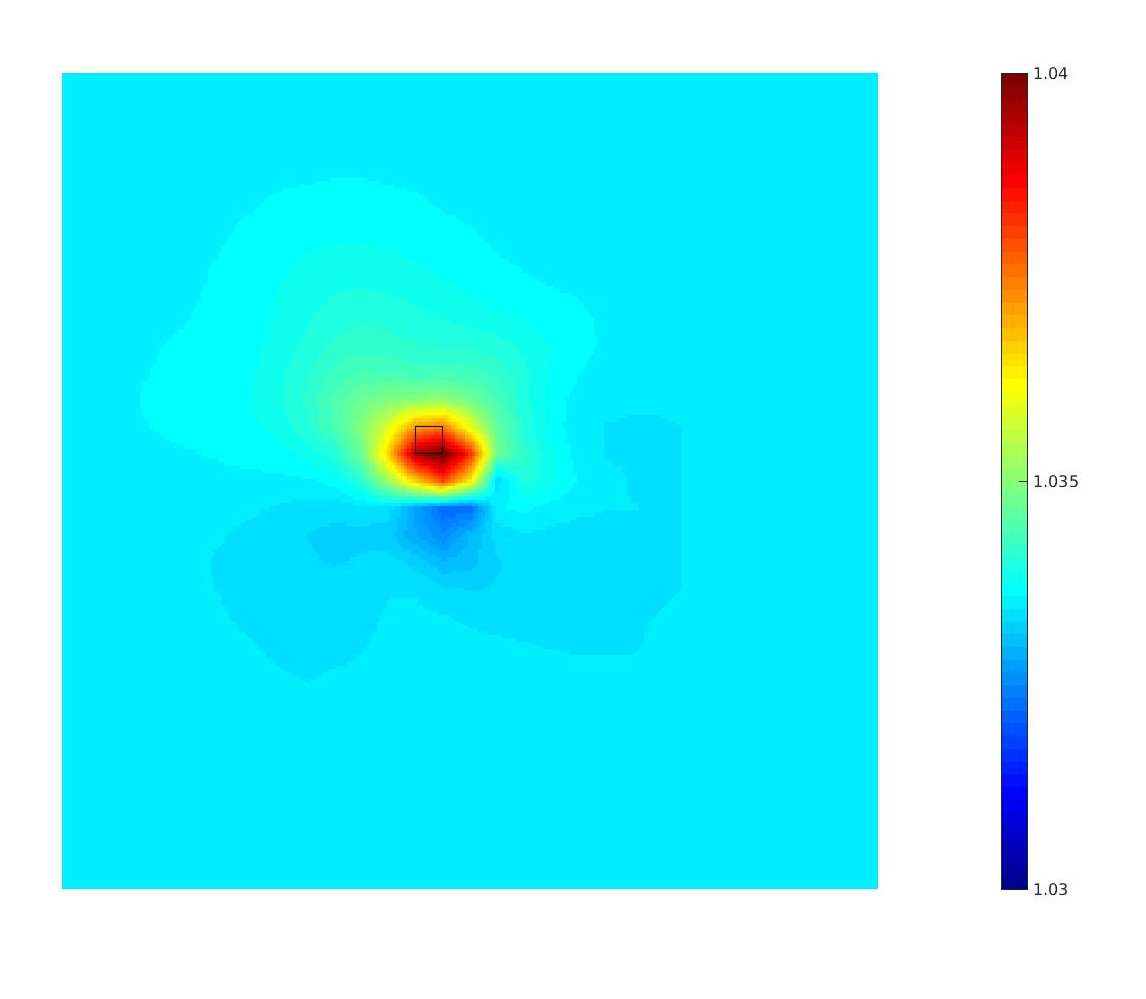}
\end{minipage}
\hspace{0.8cm}
\begin{minipage}[t]{.4\linewidth}
    \centering
    \includegraphics[width=1.0\linewidth]{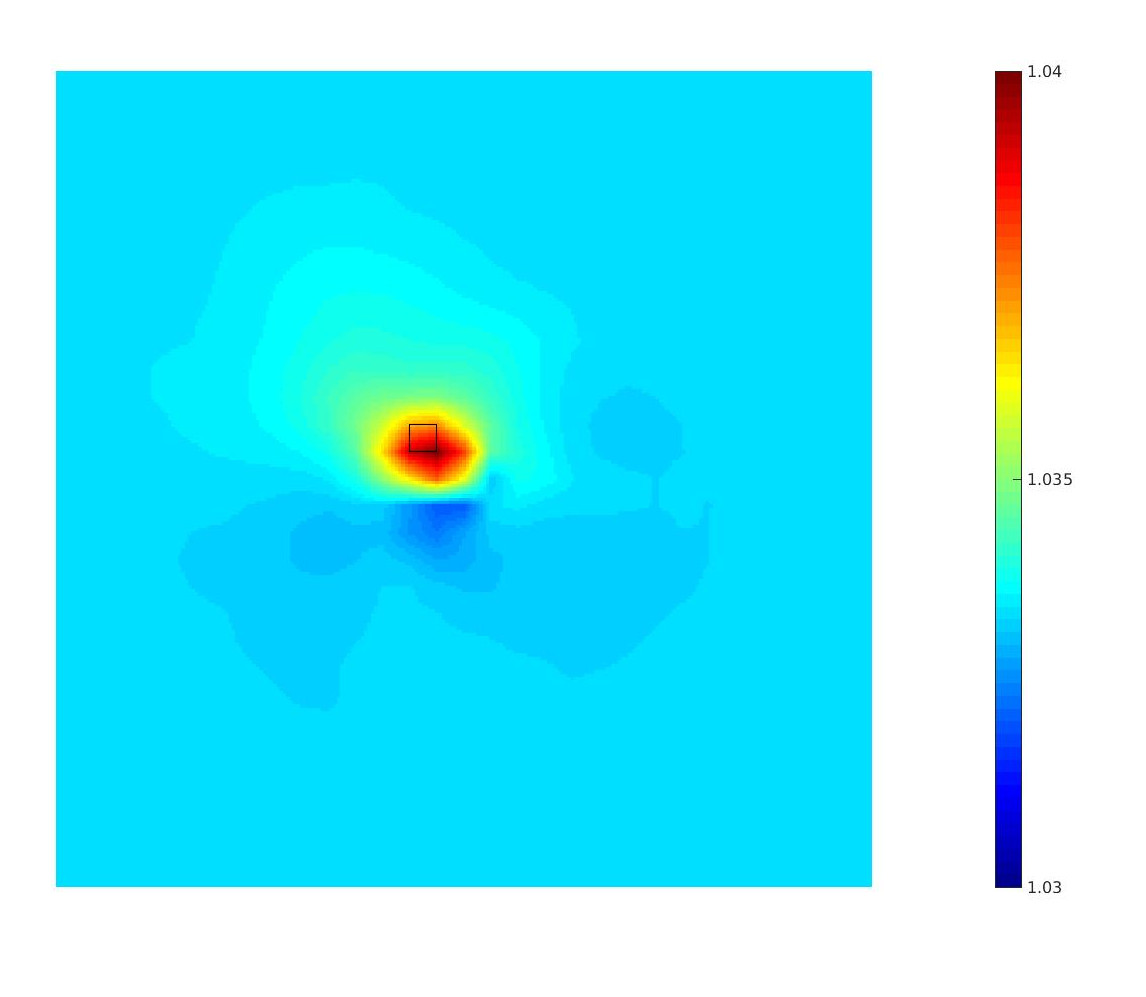}
\end{minipage}
\hspace{0.8cm}
\begin{minipage}[t]{.4\linewidth}
    \centering
    \includegraphics[width=1.0\linewidth]{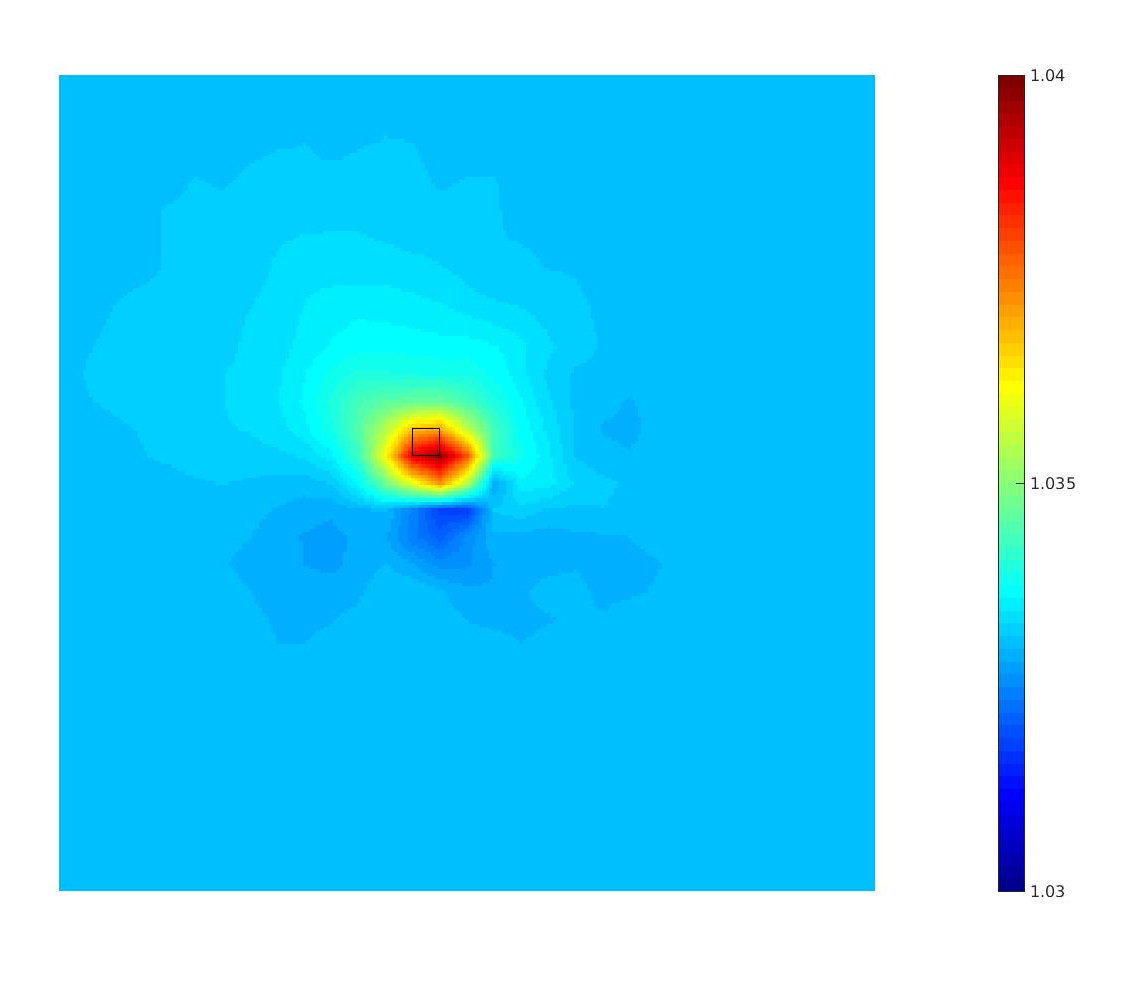}
\end{minipage}
\hspace{0.8cm}
\begin{minipage}[t]{.4\linewidth}
    \centering
    \includegraphics[width=1.0\linewidth]{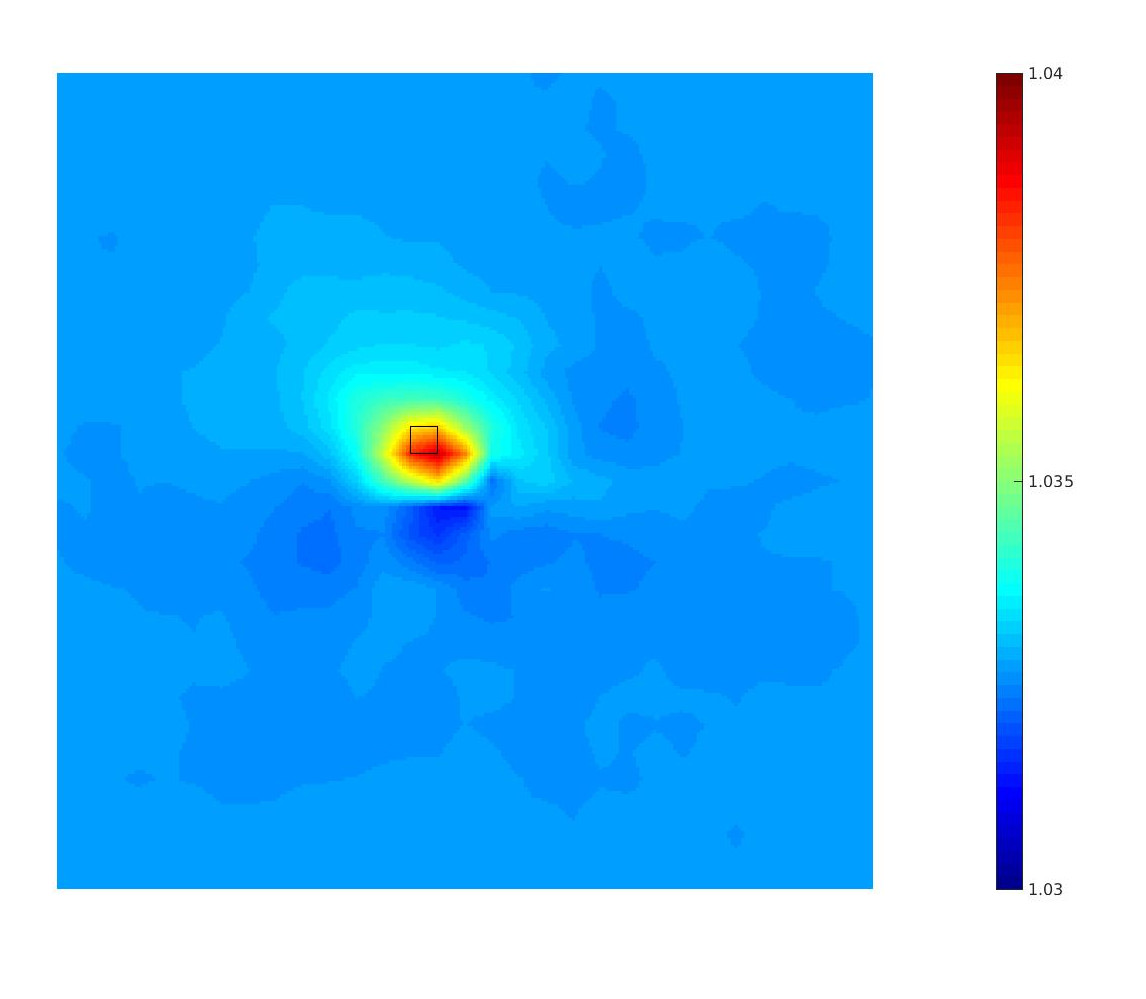}
\end{minipage}
\caption[The results of experiment 2]{The results of experiment 2. In every picture we considered damage scenario A. From top row left to bottom row right we have noise levels of $\delta=0$, $0.2$, $0.5$ and $1$.}
\label{fig:exp2}
\end{figure}
In a third experiment we consider in contrast to the first two ones incomplete data. Therefore we use as input data the displacement field $\tilde{y}=\mathcal{Q}\tilde{u}$ with the observation operator $\mathcal{Q}$. In the experiment we computed the solution with two observation operators where the sensors are attached in parallel to the four edges of the plate on both sides of the plate. For the observation operator R57d there are $57$ sensors at every edge and so all in all $448$ at the whole plate and for the other one called R8d we have $8$ sensors at 
every edge and so $56$ sensors at th whole plate. The advantage of the usage of two different observation operators, where the sensors are attached in the same way, is that we can see, how the number of sensors influences the result of the experiment. In figure \ref{fig:exp3} we can see the result after $25$ iterations. It is obviously that indeed the damage is localized in both cases, but a higher number of sensors yields a better result. Using the observation operator R8d we have more artifacts and a smaller domain such that the localized damage is less visible compared to R57d.   
\begin{figure}[ht]
\centering
\begin{minipage}[t]{.4\linewidth}
    \centering
    \includegraphics[width=1.0\linewidth]{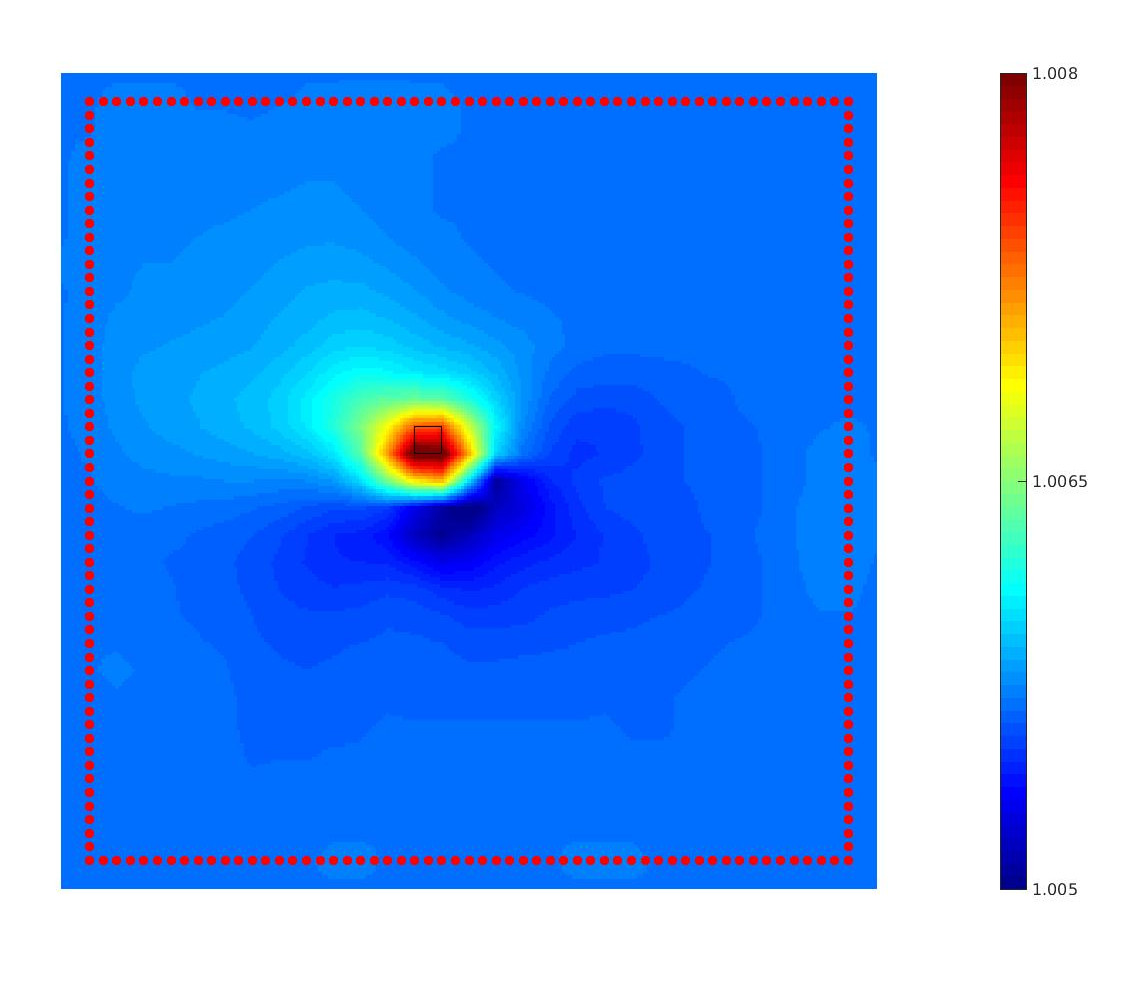}
\end{minipage}
\hspace{0.8cm}
\begin{minipage}[t]{.4\linewidth}
    \centering
    \includegraphics[width=1.0\linewidth]{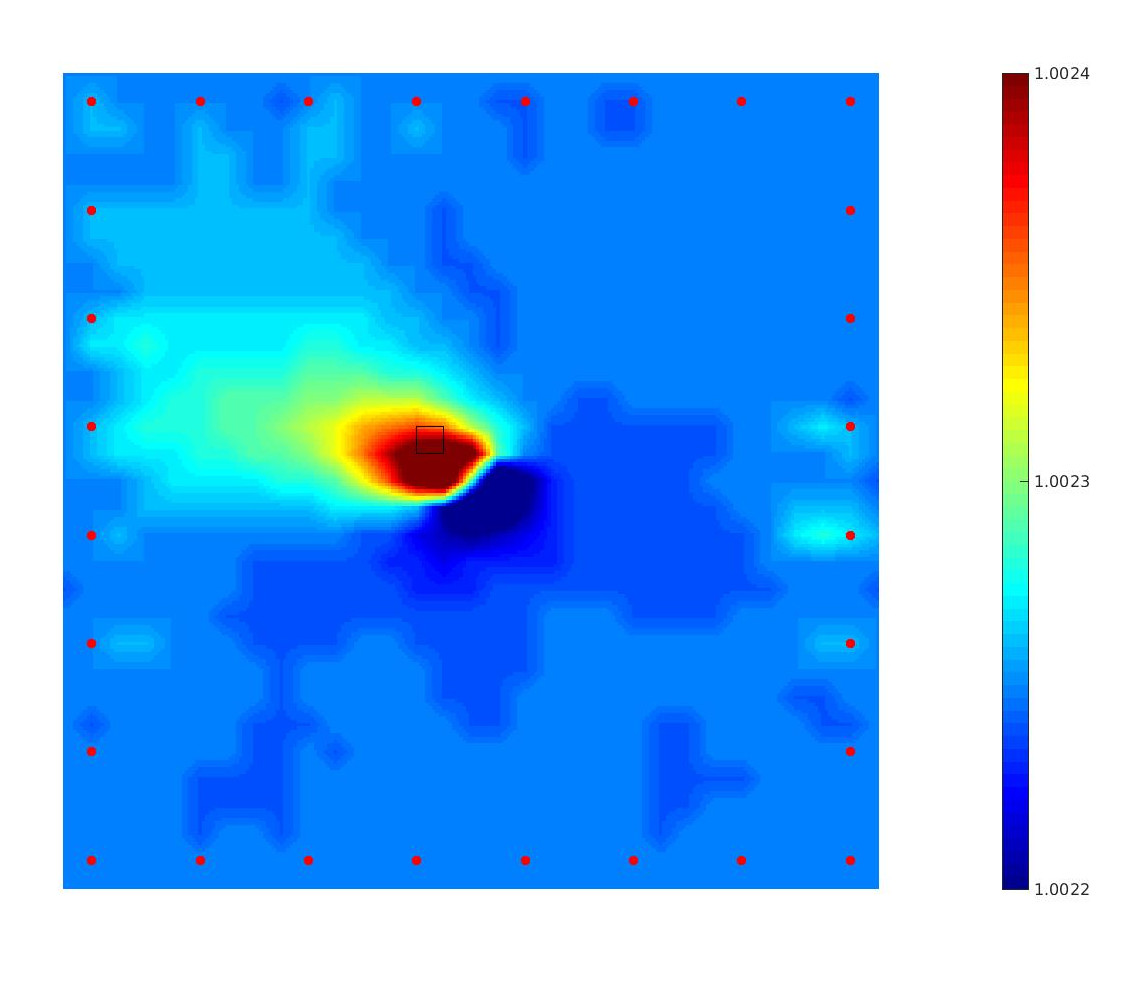}
\end{minipage}
\caption[The results of experiment 3.]{The results of experiment 3. In both pictures we considered damage scenario A with observation operator R57d (left side) and R8d (right side). The sensors are marked by red dots.}
\label{fig:exp3}
\end{figure}
\section{Conclusion}
In this article we proposed a method to detect defects in hyperelastic materials from full knowledge of the displacement field as well as sensor measurements acquired at the surface of the structure by solving a nonlinear inverse problem. The key idea 
is that we can represent the stored energy function of the hyperelastic material as a conic combination of suitable functions. Identifying these coefficients leads to the localization of defects in the structure. 
The arising inverse problem is highly nonlinear and ill-posed. In this article we solve the problem using the attenuated Landweber method. Therefore each iteration involves the solution of the nonlinear direct problem and the linear adjoint problem. The developed solver for the direct problem is also used to produce synthetic data.  
Numerical experiments showed a good performance in both cases of input data. We demonstrated that the method is also stable using noisy data. In addition it was proven that the considered identification problem fulfills the local tangential cone condition and hence that the attenuated Landweber method converges in the case of full knowledge of the displacement field to a solution of the inverse problem.

The given results show that the developed method present great space for future research. A next step in view of developing an autonomous, sensor based structural health monitoring system could be the verification of the method with real measurements at piezo sensors. Further research has to be done with respect to improve the efficiency of the method, since the Landweber method converges slowly. At the same time a finer discretization in time and space would be desirable but is out of reach because of the tremendous computing time. In that sense model reduction methods or sequential subspace optimization techniques (cf. \cite{Schopfer2009, WALD;SCHUSTER:16}) could be suitable remedies in near future.
Furthermore it might be interesting to extend the numerical considerations to curved structures such as shells as they are applied in aircraft construction. 

\bibliographystyle{siam}
\bibliography{forces-composites-bib,references-hyperelastic,references,polyconvex-references,numeriknl}
\end{document}